\documentclass[12pt,a4paper,twoside,english,american]{scrartcl}
\usepackage{helvet}
\usepackage{mathtools}
\usepackage{url}
\usepackage[T1]{fontenc}
\usepackage[latin9]{inputenc}
\usepackage{amsmath}
\usepackage{setspace}
\usepackage{amssymb}
\usepackage{enumerate}
\usepackage{pstcol}

\makeatletter

 \usepackage{amsthm}
 \theoremstyle{plain}
\newtheorem{thm}{Theorem}[section]
  \theoremstyle{plain}
  \newtheorem*{thm*}{Theorem}
  \theoremstyle{plain}
  \newtheorem{cor}[thm]{Corollary}
  \theoremstyle{remark}
  \newtheorem{rem}[thm]{Remark}
  \theoremstyle{definition}
  
 \theoremstyle{definition}
 \newtheorem*{defn*}{Definition}
  \theoremstyle{plain}
  \newtheorem{lem}[thm]{Lemma}
 \theoremstyle{definition}
  \newtheorem{example}[thm]{Example}
  \theoremstyle{remark}

\usepackage[T1]{fontenc}    

\usepackage{a4wide}        
\usepackage{tu-preprint}
\usepackage{amsmath}

\usepackage{euscript}
\usepackage{mathabx}
\allowdisplaybreaks[1] 
\usepackage{amsfonts}
\newenvironment{keywords}{ \noindent\footnotesize\textbf{Keywords and phrases:}}{}

\newenvironment{class}{\noindent\footnotesize\textbf{Mathematics subject classification 2010:}}{}

\usepackage{color,tu-preprint}

\newcommand*{\Abs}[1]{\left\lVert#1\right\rVert}
\newcommand*{\ceil}[1]{\left\lceil#1\right\rceil}
\newcommand*{\trace}{\operatorname{trace}}

\newcommand{\R}{\mathbb{R}}
\newcommand*{\Grad}{\operatorname{Grad}}

\newcommand*{\abs}[1]{\lvert#1\rvert}

\newcommand{\s}[1]{\mathcal{#1}}

\DeclareMathAccent{\Circ}{\mathalpha}{operators}{"17}

\renewcommand{\Re}{\operatorname{\mathfrak{Re}}}

\theoremstyle{plain}

\newtheorem{Sa*}[section]{Theorem}

\newtheorem{Le*}[section]{Lemma}

\newtheorem{Fo*}[subsection]{Corollary}

\newtheorem{Prop*}[section]{Proposition}

\theoremstyle{definition}
\newtheorem*{Def}{Definition}

\theoremstyle{remark}

 \numberwithin{equation}{section}

\DeclareMathOperator{\spt}{spt}

\DeclareMathOperator{\TextRe}{Re}

\renewcommand{\Re}{\TextRe}

\newcommand{\N}{\mathbb{N}}
\newcommand{\Z}{\mathbb{Z}}

\newcommand{\C}{\mathbb{C}}

\newcommand{\eps}{\varepsilon}

\DeclareMathOperator{\1}{\chi}

\newcommand{\dd}{\ \mathrm{d}}

\DeclareMathOperator{\Diverg}{Div}

\newcommand{\tor}[2]{\stackrel{#1\to #2}{\longrightarrow}}

\newcommand{\ben}{\begin{enumerate}[(i)]}
\newcommand{\een}{\end{enumerate}}

\renewcommand{\tilde}{\widetilde}
\renewcommand*{\epsilon}{\varepsilon}
\renewcommand*{\rho}{\varrho}

\makeatother

\usepackage{babel}

\begin{document}
\selectlanguage{english}%
\institut{Institut f\"ur Analysis}

\preprintnumber{MATH-AN-02-2013}

\preprinttitle{Homogenization in fractional elasticity}

\author{Marcus Waurick}

\makepreprinttitlepage

\selectlanguage{american}%
\setcounter{section}{-1}


\title{Homogenization in fractional elasticity}

\author{ Marcus Waurick\\
 Institut f\"ur Analysis, Fachrichtung Mathematik\\
 Technische Universit\"at Dresden\\
 Germany\\
 marcus.waurick@tu-dresden.de }
\maketitle
\begin{abstract} In this note we treat the equations of fractional elasticity. After establishing well-posedness, we show a compactness result related to the theory of homogenization. For this, a previous result in (abstract) homogenization theory of evolutionary equations has to be improved. The approach also permits the consideration of non-local operators (in time and space).
\end{abstract}
\setcounter{section}{-1}


\begin{class} 74Q15 (Effective constitutive equations), 35B27 (Homogenization; equations in media with periodic structure), 35Q74 (PDEs in connection with mechanics of deformable solids) \end{class}
\newpage
\tableofcontents{}

\newpage

\maketitle

\section{Introduction}

The homogenization of (stationary) elasticity has been dealt with intensively in the literature. We refer to \cite{BenLiPap,Duvaut1977,TarIntro} as general references and the quoted literature given there. The viscoelastic models, i.e., elasticity models, where the stress tensor also depends on the derivative of the strain tensor, have been studied in the light of homogenization e.g.~in \cite{Abde2009,Franc1986}. The former also incorporates time-dependent coefficients and the latter includes also thermodynamic behavior. In both the articles \cite{Abde2009,Franc1986} a memory effect is derived. Focusing on the equations of fractional elasticity, we give an operator-theoretic explanation of such effects. By framing the homogenization of fractional elasticity into the homogenization theory of evolutionary equations developed in \cite{Waurick2011Diss} with extensions in \cite{Waurick2012Asy,Waurick2012aDIE}, we show that it is -- in a particular sense -- a generic property of the homogenized equations to have memory effects. For this we have to extend results in the theory of (abstract) homogenization of evolutionary equations in the sense of \cite{Waurick2011Diss} and apply the results to the equations of fractional elasticity.


The equations of linear (fractional) elasticity can be found in \cite[Section 3.1.12.5]{Picard} or \cite[pp.~102]{DuLi}. Denoting the derivative with respect to time by $\partial_0$ and the derivative with respect to the $i$'th spatial variable by $\partial_i$, we state the equations as follows. Let $\Omega\subseteqq \R^n$ be an open set, modeling a body in its non-deformed state. We {
 describe} the displacement field of that body as a function $u\colon (0,\infty)\times \Omega \to \R^n$. The conservation of momentum yields the following equality
\begin{equation}\label{eq:1}
 \mu(x)\partial_0^2 u(t,x) - \Diverg T(t,x) = F(t,x),
\end{equation}
where $\Diverg T (t,x) = \left(\sum_{j=1}^n \partial_j T_{ij}(t,x)\right)_i$ with  $T\colon (0,\infty)\times \Omega \to \R^{n\times n}$ being the stress tensor, $\mu\colon \Omega \to \R^{n\times n}$ the mass density matrix, and $F\colon (0,\infty)\times \Omega \to \R^n$ an external forcing term. Equation \eqref{eq:1} is completed  by a constitutive relation or material law, linking the stress tensor $T$ and the strain tensor 
{\[
  \eps(u)_{i,j} \coloneqq  \frac{1}{2}\left(\partial_i u_j + \partial_j u_i \right).
\]}
In this article, we treat a constitutive relation of fractional type:
\begin{equation}\label{eq:2}
  T(t,x) = C(x) \eps(u)(t,x) + \partial_0^\alpha D(x) \eps(u)(t,x),
\end{equation}
where{\footnote{The condition that $\alpha$ only ranges from $1/2$ to $1$ is rather technical. In fact, using other methods and -- at the same time -- imposing more conditions on the operator $C$ namely, besides its boundedness, one has to assume that it is selfadjoint and non-negative definite, we have shown in \cite[Theorem 4.1]{R.Picard2012}, that the range of $\alpha$ may vary in $(0,1]$.} } $\alpha\in [1/2,1]$ is given and the maps $C\colon \Omega \to \R^{\left(n\times n\right)^2}$ and $D\colon \Omega \to \R^{\left(n\times n\right)^2}$, obeying suitable symmetry requirements, represent the coefficients of instantaneous elasticity and the effects of the memory of the material, respectively, cf.~\cite[p. 183-4]{DuLi}. { We also refer to \cite[p. 24]{Bertram}. According to \cite{Bertram} material models of fractionl type are used to describe the visco-elastic behavior with less summands in comparison to material models with only integer order derivatives.} The equations \eqref{eq:1} and \eqref{eq:2} are subject to certain boundary and initial conditions for both $u$ and $T$. 

In the first part of this note we show that an abstract well-posedness result developed in \cite{Picard} for evolutionary equations applies to these equations. We furthermore give an account of generalizations of equations \eqref{eq:1} and \eqref{eq:2}. Well-posedness conditions are the (usual) constraints of selfadjointness and uniform ellipticity and boundedness for $D$ and $\mu$. We also refer to \cite[Section 4]{Trostorff2012b} for well-posedness conditions for other viscoelastic models.

The second part of this paper deals with the issue of homogenization, i.e., we consider the equations of fractional elasticity with highly oscillatory coefficients. In order to explain the idea, we give a simple example. We assume that $C,D$ and $\mu$ are given as functions on the whole of $\R^n$. Now, what is the behavior in the limit $k\to\infty$ of the sequence of displacement fields $(u_k)_k$ being the sequence of solutions of the following problems
\begin{equation}\label{eq:3_hom}
  \begin{cases}
     \mu(k x) \partial_0^2 u_k(t,x) - \Diverg T_k(t,x) = F(t,x), \\
     T_k(t,x) = C(k x)\eps(u_k)(t,x) + \partial_0^\alpha D(k x)\eps(u_k)(t,x),
  \end{cases}\quad (k\in\N)?
\end{equation}
 We show that there exists a weakly convergent subsequence with limit $u$ satisfying the following equation of similar form
\begin{equation}\label{eq:homo_limit}
\begin{cases}
    \mu_{\textnormal{hom}} \partial_0^2 u(t,x) - \Diverg  T(t,x) = F(t,x), \\
     T(t,x) = \partial_0^\alpha D_{\textnormal{hom}}\eps(u)(t,x)+\partial_0^\alpha \sum_{k=1}^\infty\left(-\sum_{\ell=1}^\infty (-\partial_0^{-\alpha})^\ell K_\ell D_{\textnormal{hom}}\right)^{k}\eps(u)(t,x),
\end{cases}
 \end{equation}
for suitable $\mu_{\textnormal{hom}}$, $D_{\textnormal{hom}}$ and a sequence of linear operators $(K_\ell)_\ell$. One can show that in the special case of $\mu$ and $D$ periodic, the operator $\mu_{\textnormal{hom}}$ is the same as the integral mean over the period of $\mu$, and $D_\textnormal{hom}$ is the usual homogenized matrix, see e.g.~\cite{BenLiPap,TarIntro}, also cp.~Remark \ref{rem:implicitly}(f). The reason for the limit equation to contain a power series in $\partial_0^{-\alpha}$ is due to the convergence of the operators involved. The latter can only be expected to be convergent in the weak operator topology, cf.~e.g.~\cite[Proposition 4.3 and Theorem 4.5]{WaurickMMA} for easy examples.


 In order to derive such a result, we have to extend recent results (\cite{Waurick2011Diss,Waurick2012Asy}) in the theory of homogenization of evolutionary equations. In particular, in comparison to \cite{Waurick2012Asy}, we need to establish a stronger version of a 'weak-strong principle'-result. 


\section{Setting}\label{sec:set}

We recall some notions and definitions given in \cite{Picard,Kalauch}. 
Throughout this section let $H$ be a Hilbert space. For $\nu>0$, let $$H_{\nu,0}(\R;H)\coloneqq L_2(\R,\exp(-2\nu x) \dd x;H)$$	 be the Bochner-Lebesgue space of $H$-valued square-integrable functions with respect to the weighted Lebesgue measure with Radon-Nikdodym derivative $\exp(-2\nu (\cdot))$. We denote by $\partial_{0,\nu}$ the closure of the mapping\footnote{We denote by $C_c^\infty(\R;H)$ the space of indefinitely differentiable functions with values in $H$.} 
\[
  C_c^\infty (\R;H)\subseteqq H_{\nu,0}(\R;H)\to H_{\nu,0}(\R;H):f\mapsto f'. 
\]
 The \emph{Fourier-Laplace transform} 
\[
  \s L_\nu f \coloneqq \left(x\mapsto \frac{1}{\sqrt{2\pi}}\int_\R e^{-\nu y -ixy}f(y)\dd y\right) 
\]
 defined for $f\in C_c^\infty(\R;H)$ extends to a unitary mapping $H_{\nu,0}(\R;H)\to L_2(\R;H)$. We use the same notation for this extension. Denoting by $mf \coloneqq (x\mapsto xf(x))$ the multiplication-by-argument-operator on $L_2(\R;H)$ considered with its maximal domain, we get the following spectral representation for $\partial_{0,\nu}$: 
\[
  \partial_{0,\nu} = \s L_\nu^* (im+\nu)\s L_\nu.
\]
In particular, $0\in\rho(\partial_{0,\nu})$. The spectral representation gives rise to a functional calculus for the bounded and normal operator $\partial_{0,\nu}^{-1}$. More precisely, denoting $B(r,r)\coloneqq \{ z\in \C; \abs{z-r}<r\}$, we let $M: B(r,r)\to L(H)$ be a bounded and analytic mapping taking values in $L(H)$, the space of bounded linear mappings on $H$, where $r>1/(2\nu)$. Then we define
\begin{equation}\label{eq:func_calc}
   M(\partial_{0,\nu}^{-1}) \coloneqq \s L_\nu^* M\left(\frac{1}{im +\nu}\right) \s L_\nu,
\end{equation}
where $\left(M\left(\frac{1}{im +\nu}\right)\phi\right)(t)\coloneqq M\left(\frac{1}{it+\nu}\right)\phi(t)$ for $t\in \R$ and $\phi\in C_c^\infty(\R;H)$. If the value of $\nu>0$ is clear from the context, we shall neglect the explicit reference to it in the notation of the derivative. Our main example of analytic functions of $\partial_0^{-1}$ is $\partial_0^\beta\coloneqq \partial_0^{\ceil{\beta}}\partial_0^{\beta-\ceil{\beta}}=\partial_0^{\ceil{\beta}}\left(\partial_0^{-1}\right)^{\ceil{\beta}-\beta}$, for some $\beta\in\R$. The values in the space of operators come into play, when we want to consider products such as $C\partial_0^{\alpha}$ for some $C\in L(H)$ or power series expressions as in equation \eqref{eq:homo_limit}.

Now, let $A\colon D(A)\subseteqq H\to H$ be a densely defined closed linear operator with nonempty resolvent set $\rho(A)$. If $\lambda \in \rho(A)$, then we define $H_{k}(A)$ to be the completion of $D(A^{\abs{k}})$ with respect to the norm $\abs{(A-\lambda)^k\cdot}_H$ for all $k\in\Z$. Note that $H_{k}(A)$ continuously embeds into $H_{k-1}(A)$ and that $A$ has a continuous extension from $H_k(A)$ to $H_{k-1}(A)$. We will use the same notation for that extension. Furthermore, there is a canonical extension of $A$ to a densely defined closed linear operator with nonempty resolvent set in the space $H_{\nu,0}(\R;H)$ in the way that $(Au)(t):= Au(t)$ for (a.e.) $t\in \R$ and $u\in H_{\nu,0}(\R;H_1(A))$. The latter extension can be made more precise with the help of tensor product constructions, see \cite{Tensor}, \cite[Section 1.2.3]{Picard} or Section \ref{sec:wsp} in this article. { Moreover, if $C\colon D(C)\subseteqq H_0\to H_1$ is a densely defined, closed linear operator from the Hilbert space $H_0$ to $H_1$ then it has a continuous extension as an operator from $H_0(\abs{C})$ to $H_{-1}(\abs{C^*}$, see e.g.~\cite[Lemma 2.1.16]{Picard} or \cite[Remark 2.1(a)]{TroWau2012}.} 

We denote $H_{\nu,k}(\R;H)\coloneqq H_{k}(\partial_{0})$ for all $k\in\Z$. Moreover, since $0\in\rho(\partial_{0})$, we choose $\abs{\partial_{0}^k\cdot}$ as the norm in $H_{\nu,k}(\R;H)$. With the latter choice the operator $\partial_{0}^\ell$ becomes unitary from $H_{\nu,k}(\R;H)$ to $H_{\nu,k-\ell}(\R;H)$ for all $k,\ell\in \Z$, $\nu>0$.

The well-posedness result which will apply to our functional analytic formulation of fractional elasticity reads as follows.

\begin{thm}[{\cite[Theorem 6.2.5]{Picard}}]\label{thm:solth} Let $H$ be a Hilbert space, $r>0$, $\nu>\frac{1}{2r}$. Let $A\colon D(A)\subseteqq H\to H$ be skew-selfadjoint, $M\colon B(r,r)\to L(H)$ bounded and analytic. Assume that there exists $c>0$ such that for all $z\in B(r,r)$ we have
\begin{equation}\label{cmat}
 \Re z^{-1} M(z)\geqq c.
\end{equation}
 Then the closure of the operator \[
 \partial_{0} M(\partial_{0}^{-1}) + A \colon H_{\nu,1}(\R;H_{1}(A))\subseteqq H_{\nu,0}(\R;H) \to H_{\nu,0}(\R;H),
                   \]
is continuously invertible, i.e.,\[
  S_\nu \coloneqq \overline{\partial_{0} M(\partial_{0}^{-1}) + A}^{-1} \in L(H_{\nu,0}(\R;H)).
                                                         \]
Moreover, $\Abs{S_\nu}\leqq 1/c$ and $S_\nu$ is \emph{causal} in the sense that for all $a\in \R$ and $f\in H_{\nu,0}(\R;H)$ we have
\[
    \1_{(-\infty,a]}(m_0) S_\nu f = \1_{(-\infty,a]}(m_0)S_\nu \1_{(-\infty,a]}(m_0) f,
\]
where $\1_{(-\infty,a]}(m_0)$ denotes the multiplication operator in $H_{\nu,0}(\R;H)$ associated with the cut-off function $\R\ni t\mapsto \1_{(-\infty,a]}(t)$.
\end{thm}

\begin{rem}\label{rem:solth} (a) If we invoke the extrapolation spaces introduced above, it is possible to neglect the closure bar. Indeed, for $u\in H_{\nu,0}(\R;H)$, we have $Au\in H_{\nu,0}(\R;H_{-1}(A))$ and $\partial_0M(\partial_0^{-1})u\in H_{\nu,-1}(\R;H)$ and the operator $\partial_{0} M(\partial_{0}^{-1}) + A$, which is continuously invertible in $H_{\nu,0}(\R;H)$, has the (maximal) domain \[
\{ u\in H_{\nu,0}(\R;H);  \partial_{0} M(\partial_{0}^{-1})u + Au \in H_{\nu,0}(\R;H)\}.                                                                    \]
{ Moreover, since $\partial_0$ commutes with all operators in the operator sum $\partial_{0} M(\partial_{0}^{-1}) + A$, the solution theory from Theorem \ref{thm:solth} carries over to the complete extrapolation scale of $\partial_0$ in the sense that the inverse of $\partial_{0} M(\partial_{0}^{-1}) + A$ is a continuous operator from $H_{\nu,k}(\R;H)$ to $H_{\nu,k}(\R;H)$ for all $k\in \Z$.}  
 
(b) The heart of the solution theory is the positive definiteness condition \eqref{cmat}. This results in a positive definiteness condition for $\partial_{0} M(\partial_{0}^{-1}) + A$ in $H_{\nu,0}(\R;H)$. Indeed, a density argument shows that for all $u\in \{ v\in H_{\nu,0}(\R;H); \left( \partial_{0} M(\partial_{0}^{-1}) + A\right) v\in H_{\nu,0}(\R;H)\}$, we have 
\[
  c \abs{u}^2\leqq \Re\langle u,\left(\partial_{0} M(\partial_{0}^{-1}) + A\right)u\rangle. 
\]
 (c) Note that causality of the solution operator $S_\nu$ in Theorem \ref{thm:solth} relies on analyticity of $M$. Furthermore, it is possible to show that the operator $M(\partial_{0}^{-1})$ is causal itself (\cite[Theorem 6.1.1]{Picard}).

(d) Later on, we will use the following fact concerning causal operators. For all $a\in \R$, the solution operator $S_\nu$ leaves $\chi_{(a,\infty)}(m_0)[H_{\nu,0}(\R;H)]$, the Hilbert space of $H$-valued $H_{\nu,0}$-functions vanishing outside $[a,\infty)$, invariant. Indeed, let $a\in \R$ then
\begin{align*}
 \chi_{(-\infty,a]}(m_0)S_\nu\chi_{(a,\infty)}(m_0)&=\chi_{(-\infty,a]}(m_0)S_\nu\left(1-\chi_{(-\infty,a]}(m_0)\right) \\
                                            &=\chi_{(-\infty,a]}(m_0)S_\nu-\chi_{(-\infty,a]}(m_0)S_\nu\chi_{(-\infty,a]}(m_0) \\
                      			    &=\chi_{(-\infty,a]}(m_0)S_\nu\chi_{(-\infty,a]}(m_0)-\chi_{(-\infty,a]}(m_0)S_\nu\chi_{(-\infty,a]}(m_0) \\
                                            &=0.
\end{align*}
(e) The solution operator $S_\nu$ is independent of $\nu$ in the sense that if $\mu\geqq \nu$ then $S_\nu$ and $S_\mu$ coincide on the intersection of the respective domains. For a proof of this fact see \cite[Theorem 6.1.4]{Picard} or \cite[Theorem 1.4.2]{Waurick2011Diss}.
\end{rem}

\section{A functional analytic formulation for fractional elasticity}\label{sec:hom_res_appl}
 
 Let $\Omega\subseteqq \R^n$ be an open set. First, we introduce the spaces on which the spatial \mbox{(derivative-)} operators are defined.
 Denote by $H_\textnormal{sym}(\Omega)$ the subspace of $L_2(\Omega)^{n\times n}$ containing only (equivalence classes of) functions with values in the space of symmetric matrices. We endow $H_{\textnormal{sym}}(\Omega)$ with the scalar product
\[
  \left(\Phi,\Psi\right)\mapsto \int_\R \trace (\Phi(x)^*\Psi(x))\dd x
\]
to get a Hilbert space and denote by $\Grad_0$ the closure of the mapping 
\[
 L_2(\Omega)^n\supseteqq C_c^\infty(\Omega)^n\ni \Phi \mapsto \left(\frac{1}{2}\left(\partial_i \Phi_j+\partial_j\Phi_i\right)\right)_{i,j\in\{1,\ldots,n\}}\in H_{\textnormal{sym}}(\Omega),
\]
 where $\partial_i$ denotes the (distributional) derivative with respect to the $i$'th variable. It can be shown that the negative adjoint of $\Grad_0$ is the distributional divergence applied to the rows of the matrices in the space $H_\textnormal{sym}(\Omega)$, i.e.,
\[
  -\Grad_0^* \eqqcolon \Diverg \colon D(\Diverg)\subseteqq H_{\textnormal{sym}}(\Omega) \to L_2(\Omega)^n, \Phi\mapsto \left(\sum_{j=1}^n \partial_j \Phi_{ij}\right)_{i\in\{1,\ldots,n\}},
\]
where $D(\Diverg)= \left\{ \Phi\in H_{\textnormal{sym}}(\Omega); \left(\sum_{j=1}^n \partial_j \Phi_{ij}\right)_{i\in\{1,\ldots,n\}} \in L_2(\Omega)^n \right\}$. Similarly, denoting by $\Diverg_0$ the closure of the mapping
\[
  C_c^\infty(\Omega)^{n\times n}\cap H_{\textnormal{sym}}(\Omega) \subseteqq H_{\textnormal{sym}}(\Omega) \to L_2(\Omega)^n, \Phi\mapsto \left(\sum_{j=1}^n \partial_j \Phi_{ij}\right)_{i\in\{1,\ldots,n\}},
\]
 we define\footnote{We note here that for a function being in the domain of $\Grad_0$ is a generalization of vanishing at the boundary of $\Omega$, i.e., homogeneous  Dirichlet boundary conditions. To be in the domain of $\Diverg_0$ is a generalization of homogeneous Neumann boundary conditions.} $\Grad\coloneqq -\Diverg_0^*$, which is an extension of $\Grad_0$ acting as the distributional (symmetrized) gradient with maximal domain in $L_2(\Omega)^n$.

We recall the building blocks for the equations of (fractional) elasticity. 
The displacement field $u$ is related to the stress tensor $T$ via 
\begin{equation}\label{eq:displ-stress}
 \mu\partial_0^2 u - \Diverg T = F,
\end{equation}
where $\mu$ is a strictly positive definite continuous selfadjoint operator. Now, the stress tensor $T$ depends on the strain tensor $\eps(u) = \Grad u$ and its fractional time-derivative $\Grad \partial_0^{\alpha} u$ ($1/2\leqq \alpha\leqq 1$) in the way that
\begin{equation}\label{eq:strain-temp}
  T = C\Grad u + D\Grad \partial_0^\alpha u,
\end{equation}
for bounded operators $C$ and $D$, where the latter is selfadjoint and strictly positive definite. Note that $\Abs{\partial_{0}^\beta}\leqq \nu^\beta$ for all $\beta\leqq 0$. Differentiating equation \eqref{eq:strain-temp} with respect to time, introducing the unknown $v\coloneqq \partial_0 u$ and choosing $\nu$ large enough, we arrive at
\[
   \Grad v = (C+ \partial_0^\alpha D)^{-1} \partial_0 T.
 \]
Indeed, using the estimate $\Abs{\partial_{0}^\beta}\leqq \nu^\beta$ for all $\beta\leqq 0$ and the bounded invertibility of $D$, we deduce that $C+ \partial_0^\alpha D$ is continuously invertible with inverse being representable by a Neumann series expression.  { In fact, we deduce that 
\[
  \Abs{C(\partial_0^{\alpha}D)^{-1}}\leqq \Abs{C}\Abs{D^{-1}\partial_0^{-\alpha}}\leqq \Abs{C}\Abs{D^{-1}}\nu^{-\alpha}.  
\]
Thus, choosing $\nu>0$ sufficiently large, we get that $\Abs{C(\partial_0^{\alpha}D)^{-1}}<1$.} 
{ Hence, for sufficiently large $\nu>0$, } 
\begin{align*}
  \left(C+ \partial_0^\alpha D\right)^{-1}& = \left(\partial_0^\alpha D\right)^{-1}\left( C \left(\partial_0^\alpha D\right)^{-1}+1\right)^{-1} \\ 
     & = \left(\partial_0^\alpha D\right)^{-1}+\left(\partial_0^\alpha D\right)^{-1}\sum_{k=1}^\infty \left(- C \left(\partial_0^\alpha D\right)^{-1}\right)^k \\
     & = \left(\partial_0^\alpha D\right)^{-1}+\partial_0^{-2\alpha} D^{-1}\sum_{k=1}^\infty \partial_0^{-(k-1)\alpha }\left(- C D^{-1}\right)^k.
\end{align*}
We get the following system, written in block operator matrix form,
\begin{equation}\label{eq:th-el_notred}
   \left( \partial_0\begin{pmatrix} \mu & 0  \\ 0 & (C+D\partial_0^\alpha)^{-1} 
                     \end{pmatrix}
- \begin{pmatrix} 0 & \Diverg \\ \Grad & 0 
                     \end{pmatrix}\right)\begin{pmatrix} v \\ T 
\end{pmatrix} = \begin{pmatrix} F \\ 0 
\end{pmatrix}.
\end{equation}
Equation \eqref{eq:th-el_notred} is of the form stated in Theorem \ref{thm:solth}, which can be seen by setting formally
\[
   M(\partial_0^{-1})\coloneqq \begin{pmatrix} \mu & 0  \\ 0 & (C+D\partial_0^\alpha)^{-1} 
                     \end{pmatrix}
\]
and 
\[
  A\coloneqq - \begin{pmatrix} 0 & \Diverg \\ \Grad & 0 
                     \end{pmatrix}.
\]
 However, in order to obtain skew-selfadjointness of $A$, boundary conditions have to be imposed.  The first well-posedness theorem treats homogeneous Dirichlet boundary conditions and the second one is about homogeneous Neumann boundary conditions.

\begin{thm}[homogeneous Dirichlet boundary conditions]\label{thm:solth_frac_elast_D} Let $\Omega\subseteqq \R^n$ be an open set, $\alpha\in [1/2,1]$, $K>0$. Let $\mu \in L\left(L_2(\Omega)^n\right), C,D \in L\left( H_{\textnormal{sym}}(\Omega)\right)$. We assume $\mu,D$ to be selfadjoint and that there exists $c>0$ such that $\mu\geqq c$ and $D\geqq c$. Then there exists $\nu_0>0$ and  such that for all $\nu\geqq \nu_0$ and $(F,G) \in H_{\nu,0}\left(\R;L_2(\Omega)^n\oplus H_{\textnormal{sym}}(\Omega)\right)$ there is a unique $(v,T)\in H_{\nu,0}\left(\R;L_2(\Omega)^n\oplus H_{\textnormal{sym}}(\Omega)\right)$ with
\[
 \left( \partial_0\begin{pmatrix} \mu & 0  \\ 0 & (C+D\partial_0^\alpha)^{-1} 
                     \end{pmatrix}
- \begin{pmatrix} 0 & \Diverg \\ \Grad_0 & 0 
                     \end{pmatrix}\right)\begin{pmatrix} v \\ T 
\end{pmatrix} = \begin{pmatrix} F \\ G
\end{pmatrix}.
\]
Moreover, the estimate 
\[
   \abs{(v,T)}_{H_{\nu,0}\left(\R;L_2(\Omega)^n\oplus H_{\textnormal{sym}}(\Omega)\right)}\leqq K\abs{(F,G)}_{H_{\nu,0}\left(\R;L_2(\Omega)^n\oplus H_{\textnormal{sym}}(\Omega)\right)}
\]
is satisfied. 
\end{thm}
Similarly, we have a corresponding result for homogeneous Neumann boundary conditions.
\begin{thm}[homogeneous Neumann boundary conditions]\label{thm:solth_frac_elast_N} Let $\Omega\subseteqq \R^n$ be an open set, $\alpha\in [1/2,1]$, $K>0$. Let $\mu \in L\left(L_2(\Omega)^n\right), C,D \in L\left( H_{\textnormal{sym}}(\Omega)\right)$. We assume $\mu,D$ to be selfadjoint and that there exists $c>0$ such that $\mu\geqq c$ and $D\geqq c$. Then there exists $\nu_0>0$ such that for all $\nu\geqq \nu_0$ and $(F,G) \in H_{\nu,0}\left(\R;L_2(\Omega)^n\oplus H_{\textnormal{sym}}(\Omega)\right)$ we have a unique pair $(v,T)\in H_{\nu,0}\left(\R;L_2(\Omega)^n\oplus H_{\textnormal{sym}}(\Omega)\right)$ satisfying
\[
 \left( \partial_0\begin{pmatrix} \mu & 0  \\ 0 & (C+D\partial_0^\alpha)^{-1} 
                     \end{pmatrix}
- \begin{pmatrix} 0 & \Diverg_0 \\ \Grad & 0 
                     \end{pmatrix}\right)\begin{pmatrix} v \\ T 
\end{pmatrix} = \begin{pmatrix} F \\ G
\end{pmatrix}.
\]
Moreover, the estimate 
\[
   \abs{(v,T)}_{H_{\nu,0}\left(\R;L_2(\Omega)^n\oplus H_{\textnormal{sym}}(\Omega)\right)}\leqq K\abs{(F,G)}_{H_{\nu,0}\left(\R;L_2(\Omega)^n\oplus H_{\textnormal{sym}}(\Omega)\right)}
\]
is satisfied. 
\end{thm}

{ 
\begin{rem}\label{rem:dependence-of-rho} It will be obvious from the proof that the choice of $\nu_0$ in the above two Theorems can be done in a way such that $\nu_0$ only depends on $K$, the operator norms of $C,\mu,D$ and the constant of positive definiteness for $\mu$ and $D$. 
\end{rem}
}
The only difference between the Dirichlet and the Neumann case is that they have different spatial operators, i.e., in the former case we consider $\begin{pmatrix} 0 & \Diverg \\ \Grad_0 & 0 \end{pmatrix}$, whereas in the latter case we consider $\begin{pmatrix} 0 & \Diverg_0 \\ \Grad & 0  \end{pmatrix}$. The relevant common property of both these operators is that they are skew-selfadjoint. Thus, regarding the proofs of both the Theorems \ref{thm:solth_frac_elast_D} and \ref{thm:solth_frac_elast_N}, we only have to verify that the positive definiteness condition in Theorem \ref{thm:solth} is satisfied. We emphasize that we require $\mu,C,D$ to be bounded linear operators in suitable spaces. In particular, the representation as multiplication operators as in the introduction is not needed here, also cp.~Example \ref{non-local-in} below.
\begin{proof}[Proof of Theorems \ref{thm:solth_frac_elast_D} and \ref{thm:solth_frac_elast_N}] We have to verify that there exists $\nu_0>0$ such that for all $z\in B(1/(2\nu_0),1/(2\nu_0))$ we have
\begin{equation}\label{eq:wpc_ND}
   \Re \left( z^{-1} \begin{pmatrix} \mu & 0 \\ 0 & (C+z^{-\alpha}D)^{-1} \end{pmatrix} \right) \geqq 1/K.
\end{equation}
 In order to prove \eqref{eq:wpc_ND}, we observe that due to $\mu\geqq c$, $z^{-1}= it+\nu$ for some $t\in \R$ and $\nu>\nu_0$ we have that
\[
  \Re z^{-1}\mu =\Re (it+\nu)\mu=\nu \mu \geqq \nu c.
\]
Thus, choose $\nu_0$ sufficiently large to get the positive definiteness condition of the top left corner in \eqref{eq:wpc_ND}. The second diagonal entry in \eqref{eq:wpc_ND} is a bit more delicate. Let $\nu_0$ be such that $\Abs{CD^{-1}}\nu^{-\alpha}<\frac{1}{2}$ for all $\nu\geqq \nu_0$. Using a Neumann expansion, we get for $z\in B(1/(2\nu_0),1/(2\nu_0))$ that 
\[
   (C+z^{-\alpha}D)^{-1}=\left(z^{-\alpha} D\right)^{-1}+z^{2\alpha} D^{-1}\sum_{k=1}^\infty z^{(k-1)\alpha} \left(- C D^{-1}\right)^k.
\]
With the help of the estimates
\begin{align*}
  & \Abs{z^{2\alpha-1} D^{-1}\sum_{k=1}^\infty z^{(k-1)\alpha} \left(- C D^{-1}\right)^k} \\
  & \leqq \nu^{1-2\alpha}\Abs{D^{-1}\sum_{k=1}^\infty z^{(k-1)\alpha} \left(- C D^{-1}\right)^k}\\
  & \leqq \nu^{1-2\alpha} \Abs{D^{-1}C D^{-1}}\Abs{\sum_{k=0}^\infty z^{k\alpha} \left(- C D^{-1}\right)^k}\\
  & \leqq 2 \nu^{1-2\alpha} \Abs{D^{-1}C D^{-1}},
\end{align*}
 and $D^{-1}\geqq 1/\Abs{D}$ (use the spectral theorem) we get for $z\in B(1/(2\nu_0),1/(2\nu_0))$ with $z^{-1}=it+\nu$ for some $t\in \R$, $\nu>\nu_0$ that
\begin{align}
  \notag \Re z^{-1}(C+z^{-\alpha}D)^{-1} & = \Re z^{-1}\left(\left(z^{-\alpha} D\right)^{-1}+z^{2\alpha} \left(D\right)^{-1}\sum_{k=1}^\infty z^{(k-1)\alpha} \left(- C D^{-1}\right)^k\right) \\
  \label{cruc:estim}                     & \geqq \nu^{1-\alpha} /\Abs{D}- 2 \nu^{1-2\alpha} \Abs{D^{-1}C D^{-1}}.
\end{align}
 Now, we observe that by choosing $\nu_0$ large, we have that the right-hand side in \eqref{cruc:estim} is eventually larger than $1/{K}$.
\end{proof}

\begin{rem}\label{red:system}
Of course in both the Theorems \ref{thm:solth_frac_elast_D} and \ref{thm:solth_frac_elast_N} the solution operator is causal. 
\end{rem}

{ 
For later treatment of homogenization, we note here that we need a similar well-posedness result for a slightly modified system. After having proved the well-posedness of the modified system we state its interconnection to the original one. We shall only discuss the Neumann boundary value problem as the Dirichlet boundary value problem can be dealt with similarly.

\begin{thm}\label{thm:solth_mod} Let $\Omega\subseteqq \R^n$ be an open set and such that $R(\Grad)$, the range of the symmetrized gradient, is a closed\footnote{The closedness condition of the range of the symmetrized gradient happens to be the case if one assumes that the domain of $\Grad$, endowed with the graph norm, is compactly embedded in $L_2(\Omega)^n$. This, in turn, can be realized if $\Omega$ is bounded and satisfies suitable geometric requirements, see e.g.\ \cite{Weck1994} and the references therein. In the homogeneous Dirichlet case the closedness of the range is warranted if one assumes that $\Omega$ is bounded. Indeed, the closedness follows by Korn's inequality.} subset of $H_{\textnormal{sym}}(\Omega)$, $\alpha\in [1/2,1]$, $K>0$. We denote by $\Pi\colon H_{\textnormal{sym}}(\Omega)\to R(\Grad)$ the orthogonal projector onto $R(\Grad)$. Let $\mu \in L\left(L_2(\Omega)^n\right), C,D \in L\left( H_{\textnormal{sym}}(\Omega)\right)$. We assume $\mu,D$ to be selfadjoint and that there exists $c>0$ such that $\mu\geqq c$ and $D\geqq c$.  Then there exists $\nu_0>0$ such that for all $\nu\geqq \nu_0$ and $(F,\tilde G) \in H_{\nu,0}\left(\R;L_2(\Omega)^n\oplus R(\Grad)\right)$ there is a unique $(v,\tilde T)\in H_{\nu,0}\left(\R;L_2(\Omega)^n\oplus R(\Grad)\right)$ with
\[
 \left( \partial_0\begin{pmatrix} \mu & 0  \\ 0 & \left(\Pi(C+D\partial_0^\alpha)\Pi^*\right)^{-1} 
                     \end{pmatrix}
- \begin{pmatrix} 0 & \Diverg_0\Pi^* \\ \Pi\Grad & 0 
                     \end{pmatrix}\right)\begin{pmatrix} v \\ \tilde T 
\end{pmatrix} = \begin{pmatrix} F \\ \tilde G
\end{pmatrix}.
\]
Moreover, the estimate 
\[
   \abs{(v,\tilde T)}_{H_{\nu,0}\left(\R;L_2(\Omega)^n\oplus R(\Grad)\right)}\leqq K\abs{(F,\tilde G)}_{H_{\nu,0}\left(\R;L_2(\Omega)^n\oplus R(\Grad)\right)}
\]
is satisfied. 
\end{thm}
\begin{proof}
 The proof is similar to the one for the Theorems \ref{thm:solth_frac_elast_D} and \ref{thm:solth_frac_elast_N}. First of all we observe that the operator $\begin{pmatrix} 0 & \Diverg_0\Pi^* \\ \Pi\Grad & 0 \end{pmatrix}$ is skew-selfadjoint in $L_2(\Omega)^n\oplus R(\Grad)$. Observe that $\Pi \Grad$ is closed since $\Pi|_{R(\Grad)}$ is the identity. For the skew-selfadjointness it thus suffices to show that $\left(\Pi\Grad\right)^*=-\Diverg_0\Pi^*$. So, let $\Phi\in D(\left(\Pi\Grad\right)^*)$. Then for all $w\in D\left(\Pi\Grad\right)=D\left(\Grad\right)$ we have
\[
   \langle \left(\Pi\Grad\right)^*\Phi,w\rangle = \langle \Phi, \Pi\Grad w\rangle = \langle \Pi^* \Phi,\Grad w\rangle.
\]
Hence, $\Pi^*\Phi\in D(\Grad^*)=D(\Diverg_0)$ and $- \Diverg_0 \Pi^*\Phi=\left(\Pi\Grad\right)^*\Phi$. Therefore, $\left(\Pi\Grad\right)^*\subseteqq -\Diverg_0\Pi^*$. On the other hand, let $\Psi\in D(-\Diverg_0\Pi^*)$ and let $w\in D(\Grad)=D(\Pi\Grad)$. Then
\[
  \langle \Psi,\Pi\Grad w\rangle = \langle \Pi^*\Psi, \Grad w\rangle = \langle \Grad^*\Pi^*\Psi, w\rangle.
\]
Hence, $\Psi \in D((\Pi\Grad)^*)$ and $(\Pi\Grad)^*\Psi = -\Diverg_0\Pi^*\Psi$. 

Next, we need to verify that the well-posedness condition is satisfied. For this observe that the Fourier-Laplace transform commutes with both $\Pi$ and $\Pi^*$. Regarding the proof of the Theorems \ref{thm:solth_frac_elast_D} and \ref{thm:solth_frac_elast_N}, we need to inspect the expression
\[
   \left(\Pi(C+z^{-\alpha}D)\Pi^*\right)^{-1}
\]
more closely for suitable complex numbers $z$. We compute 
\[
    \left(\Pi(C+z^{-\alpha}D)\Pi^*\right)^{-1}=  \left(\Pi C\Pi^*+z^{-\alpha}\Pi D\Pi^*\right)^{-1}
\]
 Now, both $\tilde C\coloneqq \Pi C\Pi^*$ and $\tilde D\coloneqq \Pi D\Pi^*$ are continuous linear operators from $R(\Grad)$ to $R(\Grad)$. Moreover, from ${\tilde D}^* = \Pi D^*\Pi^*=\Pi D\Pi^* =\tilde D$ and 
\[
   \langle \tilde D \Phi,\Phi\rangle = \langle \Pi D \Pi^*\Phi,\Phi\rangle = \langle D \Pi^*\Phi,\Pi ^*\Phi\rangle\geqq c\abs{\Pi^*\Phi}=c\abs{\Phi}_{R(\Grad)}   \quad(\Phi \in R(\Grad))
\]
we realize that $\tilde D$ is self-adjoint and strictly positive definite in $R(\Grad)$ with $c$ as a possible positive definiteness constant. Hence, literally the same proof for the Theorems \ref{thm:solth_frac_elast_D} and \ref{thm:solth_frac_elast_N} applies here with $D$ and $C$ respectively replaced by $\tilde D$ and $\tilde C$.
\end{proof}

Having proved the well-posedness also for the projected system, we state the interconnection between the solution obtained in Theorem \ref{thm:solth_frac_elast_N} and the one in Theorem \ref{thm:solth_mod}.

\begin{thm}\label{thm:equivalence_thm} Let $\Omega\subseteqq \R^n$ be an open set and such that $R(\Grad)\subseteqq H_{\textnormal{sym}}(\Omega)$ is closed, $\alpha\in [1/2,1]$, $K>0$. We denote by $\Pi\colon H_{\textnormal{sym}}(\Omega)\to R(\Grad)$ the orthogonal projector onto $R(\Grad)$. Let $\mu \in L\left(L_2(\Omega)^n\right), C,D \in L\left( H_{\textnormal{sym}}(\Omega)\right)$. We assume $\mu,D$ to be selfadjoint and that there exists $c>0$ such that $\mu\geqq c$ and $D\geqq c$. Let $\nu_0$ be the maximum of the $\nu_0$'s occurring in Theorem \ref{thm:solth_frac_elast_N} and Theorem \ref{thm:solth_mod} and let $\nu\geqq \nu_0$ and $F\in H_{\nu,0}(\R;L_2(\Omega)^n)$.  Let $(v,T)\in H_{\nu,0}\left(\R;L_2(\Omega)^n\oplus H_{\textnormal{sym}}(\Omega)\right)$ solve
\begin{equation}\label{eq:unmod}
   \left( \partial_0\begin{pmatrix} \mu & 0  \\ 0 & (C+D\partial_0^\alpha)^{-1} 
                     \end{pmatrix}
- \begin{pmatrix} 0 & \Diverg_0 \\ \Grad & 0 
                     \end{pmatrix}\right)\begin{pmatrix} v \\ T 
\end{pmatrix} = \begin{pmatrix} F \\ 0
\end{pmatrix}
\end{equation}
 and let $(\tilde v,\tilde T)\in H_{\nu,0}\left(\R;L_2(\Omega)^n\oplus R(\Grad)\right)$ solve
\begin{equation}\label{eq:mod}
    \left( \partial_0\begin{pmatrix} \mu & 0  \\ 0 & \left(\Pi(C+D\partial_0^\alpha)\Pi^*\right)^{-1} 
                     \end{pmatrix}
- \begin{pmatrix} 0 & \Diverg_0\Pi^* \\ \Pi\Grad & 0 
                     \end{pmatrix}\right)\begin{pmatrix} \tilde v \\ \tilde T 
\end{pmatrix} = \begin{pmatrix} F \\ 0
\end{pmatrix}.
\end{equation}
Then $(\tilde v,\tilde T)=(v,\Pi T)$.
\end{thm}
\begin{proof} The second line of \eqref{eq:unmod} reads
 \[
  \Grad v=\partial_0(C+D\partial_0^\alpha)^{-1} T \in H_{\nu,-1}\left(\R;H_{\textnormal{sym}}(\Omega)\right)\cap H_{\nu,0}(\R;H_{-1}(\abs{\Diverg_0})).
 \]
Now, since $(C+D\partial_0^\alpha)^{-1}T\in H_{\nu,0}(\R;H_{\textnormal{sym}}(\Omega))$, we deduce that $\partial_0^{-1}v$ takes values in the domain of $\Grad$. Thus, we get that
 \[
  \Grad \partial_0^{-1} v=(C+D\partial_0^\alpha)^{-1} T \in H_{\nu,0}\left(\R;H_{\textnormal{sym}}(\Omega)\right).
 \] 
and after multiplying by $(C+D\partial_0^\alpha)$ and by $\Pi$ and using that $\Pi^*\Pi \Grad=\Grad$ we arrive at
 \[
  \Pi(C+D\partial_0^\alpha)\Pi^*\Pi\Grad \partial_0^{-1} v= \Pi T \in H_{\nu,-1}\left(\R;R(\Grad)\right).
 \]
Using that $\nu_0$ is chosen large enough, we get that
  \[
  \Pi\Grad \partial_0^{-1} v= \left(\Pi(C+D\partial_0^\alpha)\Pi^*\right)^{-1}\Pi T \in H_{\nu,-1}\left(\R;R(\Grad)\right).
 \]
Hence,
 \[
  \Pi\Grad v=\partial_0 \left(\Pi(C+D\partial_0^\alpha)\Pi^*\right)^{-1}\Pi T \in H_{\nu,-2}\left(\R;R(\Grad)\right)\cap H_{\nu,-1}\left(\R;H_{-1}(\abs{-\Diverg_0\Pi^*}\right).
 \]
The first line of the system \eqref{eq:unmod} reads as
\[
  -\Diverg_0 T=-\partial_0 \mu v+f \in H_{\nu,-1}(\R;L_2(\Omega)^n)\cap H_{\nu,0}(\R;H_{-1}(\abs{\Grad}))
\]
Integrating the latter equation with respect to time gives
\[
 \mu v - \Diverg_0 \partial_0^{-1}T = \partial_0^{-1} f.
\]
Observe that $\Diverg_0\Pi^*\Pi=\Diverg_0$ and, thus, $\Diverg_0 \Pi^*\Pi \partial_0^{-1}=\Diverg_0 \Pi^*\partial_0^{-1}\Pi$. Hence, 
\begin{align*}
  \partial_0^{-1} f&= \mu v- \Diverg_0 \partial_0^{-1}T \\
                   &=\mu v- \Diverg_0 \Pi^*\Pi \partial_0^{-1}T \\
                   &= \mu v- \Diverg_0 \Pi^*\partial_0^{-1} \Pi T \in H_{\nu,0}(\R;L_2(\Omega)^n).
\end{align*}
Therefore, 
\[
  f= \partial_0 \mu v- \Diverg_0 \Pi^* \Pi T \in H_{\nu,-1}(\R;L_2(\Omega)^n)
\]
We conclude that $(v,\Pi T)\in H_{\nu,0}\left(\R;L_2(\Omega)^n\oplus R(\Grad)\right)$ satisfies the system \eqref{eq:mod} as an equation in $H_{\nu,-2}(\R;L_2(\Omega)^n\oplus R(\Grad))$. Now, since the solution for \eqref{eq:mod} is unique by Theorem \ref{thm:solth_mod} (also use Remark \ref{rem:solth}(a)) it coincides with $(\tilde v,\tilde T)$.
\end{proof}}

\begin{rem}\label{ini:data}
 Recall the system \eqref{eq:th-el_notred}. We shall assume homogeneous Dirichlet boundary conditions, but certain initial conditions $(v_0,T_0)\in D(\Diverg)\oplus D(\Grad_0)$. Consider the equation
\[
  \partial_0\begin{pmatrix} \mu & 0  \\ 0 & (C+D\partial_0^\alpha)^{-1} 
                     \end{pmatrix}\left(\begin{pmatrix} v \\ T 
\end{pmatrix}-\chi_{\R_{>0}}\begin{pmatrix} v_0 \\ T_0 
\end{pmatrix}\right)
- \begin{pmatrix} 0 & \Diverg \\ \Grad_0 & 0 
                     \end{pmatrix}\begin{pmatrix} v \\ T 
\end{pmatrix} = \begin{pmatrix} F \\ 0 
\end{pmatrix}.
\]
If we assume that $F$ has support only on the positive reals, then the latter system coincides with \eqref{eq:th-el_notred} on the positive reals. Introducing the new variables $\tilde v \coloneqq v-\chi_{\R_{>0}}v_0$ and $\tilde T \coloneqq T-\chi_{\R_{>0}}T_0$ then we arrive at the following system
\[
 \left( \partial_0\begin{pmatrix} \mu & 0  \\ 0 & (C+D\partial_0^\alpha)^{-1} 
                     \end{pmatrix}
- \begin{pmatrix} 0 & \Diverg \\ \Grad_0 & 0 
                     \end{pmatrix}\right)\begin{pmatrix} \tilde v \\ \tilde T 
\end{pmatrix} = \begin{pmatrix} F \\ 0 
\end{pmatrix}+ \begin{pmatrix} 0 & \Diverg \\ \Grad_0 & 0 
                     \end{pmatrix}\chi_{\R_{>0}}\begin{pmatrix} v_0 \\ T_0 
\end{pmatrix}.
\]
The latter system admits, by our general solution theory, a unique solution with support on the positive reals due to causality. Moreover, a standard regularity argument ensures that indeed $(\tilde v,\tilde T)(t) \to 0$, i.e., $(v,T)(t)\to (v_0,T_0)$ as $t\to0+$ in a suitable sense, for details we refer to \cite[Theorem 6.2.10]{Picard}.
\end{rem}

The main observation in the last remark is that non-homogeneous initial conditions result in a different right-hand side. A similar effect occurs when considering non-homogeneous boundary conditions. We sketch the non-homogeneous Neumann boundary data case. Details can be found in \cite{Trostorff2012,Picard}. 

\begin{rem}\label{bdy:data} We impose the Neumann data $T_\partial$. In order to avoid additional (explicit) regularity assumptions on $\Omega$, we assume that the boundary data is given as a function $T_\partial \in H_{\nu,1}(\R;H_1(\abs{\Diverg}))$. Now the non-homogeneous boundary value problem associated with \eqref{eq:th-el_notred} reads as
\[
  \begin{cases}
     \partial_0 \mu v - \Diverg T = F \\
     \partial_0 T = (C+D\partial_0^\alpha)\Grad v\\
      T-T_\partial \in H_{\nu,0}(\R;H_1(\abs{\Diverg_0})).
  \end{cases}
\]
Introducing the new variable $\tilde T \coloneqq T-T_\partial$, we arrive at
\[
  \begin{cases}
     \partial_0 \mu v - \Diverg_0 \tilde T = F + \Diverg T_\partial \\
     \partial_0 \tilde T = (C+D\partial_0^\alpha)\Grad v-\partial_0 T_\partial.
  \end{cases}
\]
This gives
\begin{multline*}
  \left( \partial_0\begin{pmatrix} \mu & 0  \\ 0 & (C+D\partial_0^\alpha)^{-1} 
                     \end{pmatrix}
- \begin{pmatrix} 0 & \Diverg_0 \\ \Grad & 0 
                     \end{pmatrix}\right)\begin{pmatrix} v \\ \tilde T 
\end{pmatrix} \\ = \begin{pmatrix} F \\ 0 
\end{pmatrix} - \partial_0\begin{pmatrix} \mu & 0  \\ 0 & (C+D\partial_0^\alpha)^{-1} 
                     \end{pmatrix} \begin{pmatrix} 0 \\ -T_\partial
 \end{pmatrix} - \begin{pmatrix} 0 & \Diverg \\ \Grad & 0 
                     \end{pmatrix}  \begin{pmatrix} 0 \\ -T_\partial
 \end{pmatrix}.
\end{multline*}
\end{rem}

In summary, non-homogeneous initial or boundary data lead to different right-hand sides in the system. We end this section with an example for non-local spatial operators as coefficients. For simplicity, we consider the case of one spatial dimension, although it will be obvious how to extend the example to higher dimensions.

\begin{example}\label{non-local-in}
 We consider the case of homogeneous Dirichlet boundary conditions of a (visco-)elastic beam. In one spatial dimension ($\Omega\subseteqq \R$), the strain tensor is given by $\eps(u)=\partial_1u$. We denote the corresponding (symmetrized) gradient with Dirichlet boundary conditions by $\partial_{1,0}$. The equations read
\[
   \begin{cases} \mu \partial_0^2 u - \partial_1 T = f, \\
                  T = C\partial_{1,0} u+ D\partial_{1,0}\partial_0^\alpha u.
   \end{cases}
\]
 In order to have an example of non-local-in-space operators, note that the one-dimensional Dirichlet Laplacian $\Delta_D \coloneqq -\partial_{1,0}^*\partial_{1,0}$ is strictly negative and is clearly selfadjoint. The negative square root of $\Delta_D$ is selfadjoint and bounded (at least if $\Omega$ is assumed to be bounded). Hence, a possible model could be
\[
   \begin{cases} \mu \partial_0^2 u - \partial_1 T = f, \\
                  T = (-\Delta_D)^{-\beta}\partial_{1,0} u+ D\partial_{1,0}\partial_0^\alpha u.
   \end{cases}
\]for some $\beta>0$ and suitable $f$.
\end{example}

\section{A homogenization theorem for fractional elasticity}

Before we give the main contribution of the present paper we recall the notion of $G$-convergence:

\begin{Def}[$G$-convergence, {\cite[p.\ 74]{Gcon1}, \cite{Waurick2012aDIE}}] Let $H$ be a Hilbert space. Let $(A_n:D(A_n)\subseteqq H\to H)_n$ be a 
sequence of one-to-one mappings onto $H$ and let $B:D(B) \subseteqq H\to H$ be one-to-one.
We say that $(A_n)_n$ \emph{$G$-converges to $B$} if for all $f\in H$ the sequence $(A_n^{-1}(f))_n$ 
converges weakly to some $u$, which satisfies $u\in D(B)$ and $B(u)=f$. $B$ is called a \emph{$G$-limit} 
of $(A_n)_n$. We say that $(A_n)_n$ \emph{strongly $G$-converges to $B$, $A_n \stackrel{s-G}{\longrightarrow} B$} in $H$, if for all weakly 
converging sequences $(f_n)_n$ in $H$, we have $(A_n^{-1}(f_n))_n$ weakly converges to some $u$, which satisfies $u\in D(B)$ 
and $B(u)=\textnormal{w-}\lim_{n\to\infty}f_n$.
\end{Def}

Our main homogenization theorem in fractional elasticity reads as follows.
We will formulate both the Neumann case as well as the Dirichlet case within one theorem. Let 
\[
   A \in \left \{\begin{pmatrix} 0 & \Diverg_0 \\ \Grad & 0\end{pmatrix}, \begin{pmatrix} 0 & \Diverg \\ \Grad_0 & 0 \end{pmatrix} \right\}.
\]
Denote by $\Pi_A$ the orthogonal projection from $H_{\textnormal{sym}}(\Omega)$ onto\footnote{In the following application it is true that $R(\Grad)\subseteqq H_{\textnormal{sym}}(\Omega)$ or $R(\Grad_0)\subseteqq H_{\textnormal{sym}}(\Omega)$ is closed.} $\overline{R(\Grad)}$ if $A$ corresponds to the Neumann case or onto $\overline{R(\Grad_0)}$ if $A$ corresponds to the Dirichlet case. We define the reduced operator $A_{\textnormal{red}} \coloneqq \left(\begin{smallmatrix} 1 & 0 \\ 0 & \Pi_A \end{smallmatrix}\right)A \left(\begin{smallmatrix} 1 & 0 \\ 0 & \Pi_A^* \end{smallmatrix}\right)$. Recall that in { Theorem \ref{thm:solth_mod}} we stated a solution theory for the corresponding reduced systems of fractional elasticity.  For Hilbert spaces $H,H_1$ we write $H\hookrightarrow\hookrightarrow H_1$ if $H$ is compactly embedded in $H_1$.


\begin{thm}\label{thm:homparticular} Let $\Omega\subseteqq \R^n$ be open. Assume that either $H_1(\abs{\Grad})\hookrightarrow\hookrightarrow H_0(\abs{\Grad})$ for the Neumann case or $H_1(\abs{\Grad_0})\hookrightarrow\hookrightarrow H_0(\abs{\Grad_0})$ for the Dirichlet case. Let $(C_n)_n,(D_n)_n$ in $L(H_{\textnormal{sym}}(\Omega))$ and $(\mu_n)_n$ in $L(L_2(\Omega)^n)$ be bounded sequences. Assume there exists $c>0$ such that for all $n\in \N$ we have that $\mu_n$, $D_n$ are selfadjoint and
\[
  D_n \geqq c \text{ and } \mu_n\geqq c.
\]
Then there exists $\nu_0>0$ such that for all $\nu\geqq \nu_0$ we have a subsequence $(n_k)_k$ of $(n)_n$ and $\mu\in L(L_2(\Omega)^n)$, $D$, $(K_\ell)_\ell$ in $L(R(\Pi_A))$ such that 
\begin{multline*}
  \left( \partial_0\begin{pmatrix} \mu_{n_k} & 0  \\ 0 & (\Pi_A(C_{n_k}+D_{n_k}\partial_0^\alpha)\Pi_A^*)^{-1} 
                     \end{pmatrix}
- A_\textnormal{red}\right) \\ \stackrel{s-G}{\longrightarrow} \left( \partial_0\begin{pmatrix} \mu & 0  \\ 0 & \partial_0^{-\alpha}D^{-1}+\partial_0^{-\alpha}\sum_{\ell=1}^\infty (-\partial_0^{-\alpha})^\ell K_\ell \end{pmatrix}
- A_\textnormal{red}\right) 
\end{multline*}
in $\chi_{(a,\infty)}(m_0)[H_{\nu,0}(\R;H_0(A_{\textnormal{red}}))]$ for all $a\in\R$.
\end{thm}

\begin{rem}\label{rem:implicitly} (a) The crucial fact, why we introduced $A_{\textnormal{red}}$, is that $A_{\textnormal{red}}$ has compact resolvent. In consequence, we do not need to introduce the projection $\Pi_A$ in the case of $\Omega$ being a one-dimensional, bounded interval.
 
{ (b) From the proof of Theorem \ref{thm:homparticular}, we will see that the subsequence can be chosen in a way such that $D^{-1}=\tau_{\textnormal{w}}\textnormal{-}\lim_{k\to\infty} D_{n_k}^{-1}$ and $K_{\ell}= \tau_{\textnormal{w}}\textnormal{-}\lim_{k\to\infty} D_{n_k}^{-1}\left(C_{n_k}D_{n_k}^{-1}\right)^\ell$, $\ell\in \N$. Here $\tau_{\textnormal{w}}$ denotes the weak operator topology.}

 (c) For convenience and in order not to get drowned in notation, we stated a weaker version of what we will actually prove. Theorem \ref{thm:homparticular} follows from Corollary \ref{thm:acr} being itself a consequence of Theorem \ref{thm:ahr} below. In particular, in the situation of Theorem \ref{thm:homparticular}, we will show that for the extracted subsequence $(n_k)_k$ we have the following result. For $\nu_0$ sufficiently large, let $\nu_1>\nu\geqq \nu_0$ and take a weakly convergent sequence \[
(f_k,g_k)_k \textnormal{ in }H_{\nu,0}(\R;H_0(A_{\textnormal{red}}))\cap H_{\nu_1,0}(\R;H_0(A_{\textnormal{red}}))\eqqcolon H_{\nu,0}\cap H_{\nu_1,0}(\R;H_0(A_{\textnormal{red}}))                                                                                                                                             \]
with limit $(f,g)$ satisfying the following equi-integrability condition at $-\infty$:
\[
  \sup_{k\in\N} \abs{\chi_{(-\infty,a)}(m_0)(f_k,g_k)}_{\nu_1,0}\to 0 \ (a\to-\infty).
\]
Denote \[
                              S_k\coloneqq \left( \partial_0\begin{pmatrix} \mu_{n_k} & 0  \\ 0 & (\Pi_A(C_{n_k}+D_{n_k}\partial_0^\alpha)\Pi_A^*)^{-1} 
                     \end{pmatrix}
- A_\textnormal{red}\right)^{-1} \quad (k\in\N)
                             \]
and $S$ correspondingly. Then $S_k (f_k,g_k)\rightharpoonup S(f,g)$ in $H_{\nu_1,0}(\R;H_0(A_{\textnormal{red}}))$ as $k\to\infty$. 

The last assertion is stronger than the assertion in Theorem \ref{thm:homparticular}. Indeed, Theorem \ref{thm:homparticular} states that for all weakly convergent sequences $(f_k,g_k)_k$ in $H_{\nu,0}(\R;H_0(A_{\textnormal{red}}))$ consisting of functions supported on $(a,\infty)$ for some $a\in \R$, we get that $S_k (f_k,g_k)\rightharpoonup S(f,g)$ in $H_{\nu,0}(\R;H_0(A_{\textnormal{red}}))$. In order to deduce the last convergence from the above, observe that by the condition on the support of the $f_k$'s and $g_k$'s the equi-integrability condition is trivially satisfied. Moreover, $(f_k,g_k)_k$ is weakly convergent in $H_{\nu_1,0}(\R;H_0(A_{\textnormal{red}}))$ for all $\nu_1>\nu$ also by the support condition. By the above, we infer that $S_k(f_k,g_k)\rightharpoonup S(f,g)$ in $H_{\nu_1,0}(\R;H_0(A_{\textnormal{red}}))$. In particular, this means that $\langle S_k(f_k,g_k),\phi\rangle\to \langle S(f,g),\phi\rangle$ as $k\to\infty$ for all $\phi\in C_c^\infty(\R;H_0(A_{\textnormal{red}}))$. The boundedness of $(S_k(f_k,g_k))_k$ in $H_{\nu,0}(\R;H_0(A_{\textnormal{red}}))$ yields the asserted convergence.

(d) We also note that from Theorem \ref{thm:homparticular} it follows that
   \begin{multline*}
  \left( \partial_0\begin{pmatrix} \mu_{n_k} & 0  \\ 0 & (\Pi_A(C_{n_k}+D_{n_k}\partial_0^\alpha)\Pi_A^*)^{-1} 
                     \end{pmatrix}
- A_\textnormal{red}\right) \\ \stackrel{G}{\longrightarrow} \left( \partial_0\begin{pmatrix} \mu & 0  \\ 0 & \partial_0^{-\alpha}D^{-1}+\partial_0^{-\alpha}\sum_{\ell=1}^\infty (-\partial_0^{-\alpha})^\ell K_\ell \end{pmatrix}
- A_\textnormal{red}\right) 
\end{multline*}
in $H_{\nu,0}(\R;H_0(A_{\textnormal{red}}))$. Indeed, since $C_c^\infty(\R;H_0(A_{\textnormal{red}}))$ is dense in $H_{\nu,0}(\R;H_0(A_\textnormal{red}))$, the assertion follows with a standard density argument (use that $(S_k)_k$ is bounded).

 (e) Implicitly Theorem \ref{thm:homparticular} asserts that the limiting equation is well-posed. Indeed, we will show that the norm bound of the solution operator is the same. (This follows easily from the abstract homogenization result Corollary \ref{thm:acr}) 

(f) We emphasize that even if both $(D_n)_n$ and $(D_n^{-1})_n$ converge in the weak operator topology, it is unclear whether $((\Pi_A D_n \Pi_A^*)^{-1})_n$ converges in the weak operator topology of $L(R(\Pi_A))$, see also \cite[Remark 4.6(ii)]{Waurick2012Asy}. However, it is possible to show that if $(D_n)_n$ is a multiplication operator, i.e., $D_n = D(n\cdot)$ for some periodic, bounded and measurable mapping $D\colon \R^n\to\R^{(n\times n)^2}$ as in classical homogenization theory, then $((\Pi_A D_n \Pi_A^*)^{-1})_n$ converges in the weak operator topology. The limit is well-known and can be computed by solving so-called \emph{local} problems
, cf.\ e.g.\ \cite{BenLiPap,CioDon,TarIntro}. In this way, we extend well-known homogenization results.  
\end{rem}

\begin{rem}\label{rem:Abde} In \cite{Abde2009}, a Kelvin-Voigt model is treated. The space-time operator of this equation may be reformulated as follows
\[
   \partial_0 \begin{pmatrix} \rho_n & 0 \\ 0 & (\partial_0 B_n + A_n)^{-1} \end{pmatrix}  - A_{\textnormal{red}}, 
\]
where $\rho_n$, $A_n$, $B_n$ are suitable coefficients. The coefficients $A_n$ and $B_n$ in the situation of \cite{Abde2009} may also depend on time. To get that the corresponding equation is well-posed, it is also assumed that $B_n$ is selfadjoint and strictly positive definite uniformly in $n$ as well as differentiable w.r.t.~time with bounded and measurable derivative. For $A_n$ a similar differentiability condition is assumed.  The authors of \cite{Abde2009} show that the homogenized system, i.e., the limit as $n\to\infty$, has memory effects. Note that, in the time-independent framework treated here -- as it was already observed (at least) in \cite{Franc1986} -- the occurrence of memory effects is clear. We can give a heuristic operator-theoretic explanation for the memory effect phenomenon to occur. Indeed, in order to express the limit of  $(\partial_0B_n + A_n)^{-1}$ choose $\nu$ large enough to derive a Neumann series expansion. Considering the resulting series $\sum_{k=0}^\infty (-B_n^{-1}A_n\partial_0^{-1})^{k}\partial_0^{-1}B_n^{-1}$, we may let $n$ tend to infinity in this expression (possibly by passing to a subsequence). The limit series is then of the form  $\sum_{k=0}^\infty C_k(\partial_0^{-1})^{k}\partial_0^{-1}$ for suitable $(C_k)_k$. In general the operators involved only converge in the weak operator topology. Since computing the inverse is not a continuous process in the weak operator topology, one cannot expect that the limit series is a Neumann series expression similar to the series one started out with. This fact then results in a memory effect. The authors of \cite{Abde2009,Franc1986} give also explicit formulas for the limit equation. We also refer to \cite[Example 5.7]{WaurickMMA} for a similar effect in the homogenization of ordinary differential equations. 
\end{rem}

%
%
%

\begin{rem}\label{rem:homeqns} The homogenized equations written in a similar form as the system one started out with are given in the introduction. We will consider the homogeneous Dirichlet case only. The result of Theorem \ref{thm:homparticular} roughly states the following. Let $a\in\R$ and let $(f_k)_k$ be a weakly convergent sequence in $H_{\nu,0}(\R;L_2(\Omega)^n)$ with limit $f$ and such that\footnote{We denote the support of $v\in H_{\nu,0}(\R;H)$ by $\spt v\coloneqq \R\setminus \bigcup\{ U\subseteqq \R; v|_U=0, U \text{ open}\}$.} $\inf_{k\in\N}\inf\left(\spt f_k \right)\geqq a$. Let $(u_k,T_k)_k$ in $H_{\nu,1}(\R;L_2(\Omega)^n)\oplus H_{\nu,0}(\R;R(\Grad_0))$ be the sequence of solutions of the equations
\begin{multline*}
  \begin{cases}
   \mu_{n_k}\partial_0^2 u_k(t,x) -\Diverg\Pi_A^* T_k(t,x) = f_k(t,x)\\
   \Pi_AT_k(t,x)=\Pi_AC_{n_k}\Pi_A^*\Pi_A\Grad_0 u_k(t,x)+\Pi_A\partial_0^\alpha D_{n_k}\Pi_A^*\Pi_A\Grad_0 u_k (t,x), 
  \end{cases}\\ \quad ((t,x)\in \R\times\Omega, k\in\N). 
\end{multline*}
Then $(u_k,T_k)_k$ weakly converges to the solution $(u,T)$ of the following system
\[
 \begin{cases}
   \mu\partial_0^2 u(t,x) -\Diverg\Pi_A^* T(t,x) = f(t,x)\\
   T(t,x)=\left(\partial_0^{-\alpha}D^{-1}+\partial_0^{-\alpha}\sum_{\ell=1}^\infty (-\partial_0^{-\alpha})^\ell K_\ell\right)^{-1}\Pi_A\Grad_0 u(t,x), 
  \end{cases}\quad ((t,x)\in \R\times\Omega).
\]
Another Neumann expansion of the inverse in the latter equation gives the system in the introduction. In principle, we could discuss non-homogeneous boundary value problems here. However, as they can be discussed in a unified and abstract way as in Section \ref{sec:abs_comp}, see e.g.~\cite{Trostorff2012}, we postpone the discussion to a future article. 

{ We emphasize here that our homogenization theorem only asserts something about the convergence of the part of the stress T, which can be represented as a symmetrized gradient. However, in view of Theorem \ref{thm:equivalence_thm}, we see that the homogenization theorem presented asserts the convergence of (a subsequence of) the displacement fields solving the equation for fractional elasticity.}
\end{rem}

We apply our findings to Example \ref{non-local-in}:
\begin{example}[Example \ref{non-local-in} continued]\label{non-local-2} Assume that $\Omega=(0,1)$. By the Arzela-Ascoli theorem, $H_1(\abs{\partial_{1,0}})\hookrightarrow\hookrightarrow H_0(\abs{\partial_{1,0}})$. Assume $\mu,D$ are given as $1$-periodic, bounded, measurable functions from $\R$ to $\R$. Moreover, we assume that $\mu,D\geqq c$. By \cite[Theorem 2.6]{CioDon}, we deduce that $(\mu(n\cdot))_n$ and $(D(n\cdot)^k)_n$ converge for any $k\in\N$ to the respective integral means over the period in the $\sigma(L_\infty,L_1)$-topology. Thus, we infer the convergence of the induced multiplication operators in the weak operator topology (compare with \cite[Proposition 4.3]{WaurickMMA}). Let $(u_k,T_k)_k$ be the sequence of solutions to the following problems
\[
  \begin{cases}
   \mu(k\cdot )\partial_0^2 u_k -\partial_1T_k = f_k\\
   T_k=\left(-\Delta_D\right)^{-\beta_k}\partial_{1,0} u_k+D(k\cdot)\partial_0^\alpha \partial_{1,0} u_k, 
  \end{cases} \quad (k\in\N), 
 \]where $(f_k)_k$ is a weakly convergent sequence in an appropriate space and $(\beta_k)_k$ a convergent sequence of non-negative real numbers. Denoting $f\coloneqq \textnormal{w-}\lim_{k\to\infty}f_k$, we deduce that (a subsequence of) $(u_k,T_k)_k$ weakly converges to the solution of the equation
\[
 \begin{cases}
   {\displaystyle \left(\int\limits_{[0,1]}\mu(x)\dd x\right) \partial_0^2 u -\partial_1T = f}\\{\displaystyle
   T=\left(\int\limits_{[0,1]}\frac{1}{D(x)}\dd x\right)^{-1}\partial_0^\alpha\partial_{1,0}u+\partial_0^\alpha\sum\limits_{k=1}^\infty\left(-\sum\limits_{\ell=1}^\infty (-\partial_0^{-\alpha})^\ell K_\ell\left(\int\limits_{[0,1]}\frac{1}{D(x)}\dd x\right)^{-1}\right)^{k}\partial_{1,0}u,}
   \end{cases}
\]where { $K_\ell = \lim_{k\to \infty} D(k\cdot)^{-1}\left(D(k\cdot)^{-1}(-\Delta_D)^{-\beta_k}\right)^\ell$} (for a suitable subsequence).
\end{example}

{ A combination of the results presented yields -- to the best of the author's knowledge -- the first explicit formula for the limit equation in fractional elasticity:}

{ \begin{thm}\label{thm:hom-1-dimensional} Let $\mu,D,c \in L_\infty(\R)$, $\alpha\in [1/2,1]$. Assume that $\mu,D\geqq d$ for some $d>0$ and that $\mu,D,c$ are periodic with period $1$. Then there exists $\nu_0>0$ such that for all $\nu>\nu_0$ and for all $f\in H_{\nu,0}(\R;L_2(0,1))$,  $k\in\N$ there exist uniquely determined $(u_k,T_k)\in H_{\nu,1}(\R;L_2(0,1))\oplus H_{\nu,0}(\R;L_2(0,1))$ such that\footnote{By $\psi(m_1)$ we denote the multiplication operator of multiplying with a function $\psi$ in $L_2(0,1)$.}
        \[
          \begin{cases} \mu(k m_1) \partial_0^2 u_k - \partial_1 T_k = f, \\
                  T_k = c(k m_1)\partial_{1,0} u_k+ D(k m_1)\partial_{1,0}\partial_0^\alpha u_k.
   \end{cases}
        \]
Moreover, the sequence $((u_k,T_k))_k$ is weakly convergent. The respective limit $(u,T)\in H_{\nu,1}(\R;L_2(0,1))\oplus H_{\nu,0}(\R;L_2(0,1))$ satisfies
\[
 \begin{cases}{\displaystyle
   \left(\int\limits_{[0,1]}\mu(x)\dd x\right) \partial_0^2 u -\partial_1T = f}\\ {\displaystyle
   T=\left(\int\limits_{[0,1]}\frac{1}{D(x)}\dd x\right)^{-1}\partial_0^\alpha\partial_{1,0}u} \\ {\displaystyle \quad\quad  +\partial_0^\alpha\sum\limits_{k=1}^\infty\left(-\sum\limits_{\ell=1}^\infty (-\partial_0^{-\alpha})^\ell \int\limits_{[0,1]}\left(\frac{c(x)^\ell}{D(x)^{\ell+1}}\right)\dd x\left(\int\limits_{[0,1]}\frac{1}{D(x)}\dd x\right)^{-1} \right)^{k}\partial_{1,0}u.}
   \end{cases}
\]
\end{thm}
\begin{proof}[Proof of Theorem \ref{thm:hom-1-dimensional} -- part 1: Well-posedness]
  For the unique existence of $(u_k,T_k)$, $k\in\N$, we apply Theorem \ref{thm:solth_frac_elast_D}. In fact, the conditions $\mu,D\geqq d$ implies that the multiplication operators $\mu(km_1)$ and $D(km_1)$ are selfadjoint and strictly positive definite in $L_2(0,1)$ with positive definiteness constants independent of $k$. The boundedness of $\mu,D,c$ yields that $\mu(km_1),D(km_1),c(km_1)\in L(L_2(0,1))$ with operator norms independent of $k$. Hence, by Theorem \ref{thm:solth_frac_elast_D} together with Remark \ref{rem:dependence-of-rho} (choose $K=1$), there exists $\nu_0>0$ such that  $\nu\geqq \nu_0$ and for all $(F,G)\in H_{\nu,0}(\R;L_2(0,1)^2)$ there is a unique $(v,T)\in H_{\nu,0}(\R;L_2(0,1)^2)$ with 
\[
 \left( \partial_0\begin{pmatrix} \mu(km_1) & 0  \\ 0 & (c(km_1)+D(km_1)\partial_0^\alpha)^{-1} 
                     \end{pmatrix}
- \begin{pmatrix} 0 & \partial_1 \\ \partial_{1,0} & 0 
                     \end{pmatrix}\right)\begin{pmatrix} v \\ T 
\end{pmatrix} = \begin{pmatrix} F \\ G
\end{pmatrix}.
\]
Now, for the special choice $(F,G)=(f,0)$ for some $f\in H_{\nu,0}(\R;L_2(0,1))$, we conclude unique existence of the corresponding solutions $((v_k,T_k))_k$. Using the extrapolation spaces of $\partial_0$ and $\begin{pmatrix} 0 & \partial_1 \\ \partial_{1,0} & 0 \end{pmatrix}$, we may read the equation satisfied by $(v_k,T_k)$ line-by-line in order to get that
\[
 \partial_0 \mu(km_1) v_k - \partial_{1} T_k = f
\]
and
\[
  \partial_0(c(km_1)+D(km_1)\partial_0^\alpha)^{-1} T_k-\partial_{1,0}v_k = 0. 
\]
Thus,
\[  T_k=(c(km_1)+D(km_1)\partial_0^\alpha)\partial_{1,0}\partial_0^{-1}v_k.\]
Hence, setting $u_k\coloneqq \partial_0^{-1}v_k \in H_{\nu,1}(\R;L_2(0,1))$, we arrive at
\[
          \begin{cases} \mu(k m_1) \partial_0^2 u_k - \partial_1 T_k = f, \\
                  T_k = c(k m_1)\partial_{1,0} u_k+ D(k m_1)\partial_{1,0}\partial_0^\alpha u_k.
   \end{cases}
\]
On the other hand, performing the computations backwards, we realize that any solution $(u_k,T_k)$ of the latter equation yields a solution $(\partial_0 u_k,T_k)$. Hence, the part of uniqueness and existence in the theorem in question is proved. 
\end{proof}}

In order to prove { the Theorems \ref{thm:homparticular} and \ref{thm:hom-1-dimensional}}, we have to develop some abstract theory in homogenization. This will be done in the next section.

\section{An abstract compactness result in the theory of homogenization}\label{sec:abs_comp}

Regarding the examples already discussed in earlier work, e.g.\ in \cite{Waurick2011Diss,Picard}, we realize that in order to model homogenization theory in the abstract setting of Theorem \ref{thm:solth} one should focus on the bounded and analytic function $M$, since it encodes the properties of the underlying material. That is why we introduce the Hardy space
\[
  \s H^\infty(E;L(H)) \coloneqq \{ M \colon E\to L(H); M \text{ analytic, bounded}\},
\]
 where $E\subseteqq \C$ is an open subset. $\s H^\infty(E;L(H))$, or $\s H^\infty$ for short if $E$ and $H$ is clear from the context, becomes a Banach space if endowed with the norm $M\mapsto \Abs{M}_\infty\coloneqq \sup_{z\in E} \Abs{M(z)}$. Let $r>0$. We call elements $M\in \s H^\infty(B(r,r);L(H))$ also \emph{material laws} or \emph{constitutive relations}. If, in addition, $M$ satisfies the estimate \eqref{cmat} for some $c>0$, we call $M$ a \emph{$(c)$-material law}. Homogenization theory goes along with the study of particular weak topologies, cp.\ e.g.\  \cite{TarIntro,Well,Waurick2011Diss}. In view of applications to homogenization theory the topology induced by the norm on $\s H^\infty$ is too strong. Hence, we need to introduce a weaker one. Therefore we denote $\s H(E)\coloneqq \{ f\colon  E\to\C; f\text{ analytic}\}$ and endow this space with the \emph{compact open topology}, i.e., the topology induced by uniform convergence on compact sets. Now, let $\tau_{\textnormal{w}}$ be the topology on $\s H^\infty(E;L(H))$ induced by the mappings
\[
   \s H^\infty \ni M\mapsto \left( E\ni z\mapsto \left\langle \phi, M(z)\psi \right\rangle\right) \in \s H(E)
\]
for $\phi,\psi\in H$. We denote $\s H^\infty_{\textnormal{w}}\coloneqq \left( \s H^\infty,\tau_{\textnormal{w}}\right)$ and identify $\s H^\infty_{\textnormal{w}}$ with the underlying set if we want to emphasize the topology under consideration. A subset of $\s H^\infty_{\textnormal{w}}$ is called \emph{bounded}, if it is bounded with respect to $\Abs{\cdot}_\infty$.  The main abstract homogenization theorem proved here reads as follows. 

\begin{thm}[abstract homogenization result]\label{thm:ahr} Let $H$ be a Hilbert space, $A\colon D(A)\subseteqq H\to H$ skew-selfadjoint with compact resolvent, $\nu>0$, $r>1/(2\nu)$. Let $(M_n)_n$ be a bounded and convergent sequence in 
\[
 \s H^{\infty,c}_{\textnormal{w}}(B(r,r);L(H)) \coloneqq \{ M \in \s H^\infty_\textnormal{w}; \forall z\in B(r,r):\Re z^{-1}M(z)\geqq c\}. 
\]
Denote $M\coloneqq \lim_{n\to\infty} M_n$.
Then for all weakly convergent $(f_n)_n$ in $H_{\nu,0}\cap H_{\nu_1,0}(\R;H)$ for some $\nu>\nu_1>1/(2r)$, which satisfy
\[
  \sup_{n\in\N} \abs{\chi_{(-\infty,a]}(m_0)f_n}_{\nu,0} \to 0 \quad (a\to-\infty)
\]
the solutions $(u_n)_n$ in $H_{\nu,0}(\R;H)$ of the problems
\[
  \left(\partial_0 M_n(\partial_0^{-1})+A\right) u_n = f_n
\]
weakly converge to the solution $u$ of the problem
\[
  \left(\partial_0 M(\partial_0^{-1})+A\right) u = \textnormal{w-}\lim_{n\to\infty}f_n.
\]
\end{thm}

We will prove Theorem \ref{thm:ahr} in Section \ref{sec:wsp}.

Note that since $(M_n)_n$ converges in $\s H_{\textnormal{w}}^{\infty,c}$ the limit $M$ is also a $(c)$-material law. Hence the limit equation is well-posed. Moreover, note that, by \cite[Proposition 1.3]{Waurick2012Asy}, which asserts that $\s H^{\infty,c}_{\textnormal{w}}\subseteqq \s H^{\infty}_{\textnormal{w}}$ is closed, convergence for $(M_n)_n$ in the subspace $\s H^{\infty,c}_{\textnormal{w}}$ or in $\s H^{\infty}_{\textnormal{w}}$ is equivalent, also see Remark \ref{rem:implicitly}(e).

What we actually use for the proof of Theorem \ref{thm:homparticular} is the following statement.

\begin{cor}[abstract compactness result]\label{thm:acr} Let $H$ be a Hilbert space, $A\colon D(A)\subseteqq H\to H$ skew-selfadjoint with compact resolvent, $\nu>0$, $r>1/(2\nu)$. Let $(M_n)_n$ be a bounded sequence in $\s H^{\infty,c}_{\textnormal{w}}(B(r,r);L(H))$. Then there exists a subsequence $(n_k)_k$ and $M\in  \s H^{\infty,c}_\textnormal{w}$ such that $M=\lim_{k\to\infty} M_{n_k}$ and for all weakly convergent $(f_k)_k$ in $H_{\nu,0}\cap H_{\nu_1,0}(\R;H)$ for some $\nu>\nu_1>1/(2r)$, which satisfy
\[
  \sup_{k\in\N} \abs{\chi_{(-\infty,a]}(m_0)f_k}_{\nu,0} \to 0 \quad (a\to-\infty)
\]
the solutions $(u_k)_k$ in $H_{\nu,0}(\R;H)$ of the problems
\[
  \left(\partial_0 M_{n_k}(\partial_0^{-1})+A\right) u_k = f_k
\]
weakly converge to the solution $u$ of the problem
\[
  \left(\partial_0 M(\partial_0^{-1})+A\right) u = \textnormal{w-}\lim_{k\to\infty}f_k.
\]
%
\end{cor}

Corollary \ref{thm:acr} is indeed a straightforward consequence of Theorem \ref{thm:ahr}, once the following compactness result is shown. A similar result has already been pointed out in \cite[Theorem 3.4]{WaurickMMA}. Since we assert a slightly stronger version of \cite[Theorem 3.4]{WaurickMMA}, we present a (new) short proof in this exposition.

\begin{thm}\label{thm:compactnessofmatlaw} Let $H$ be a Hilbert space, $E\subseteqq \C$ open. Then
\[
  B_{\s H^\infty} \coloneqq \{ M \in \s H^\infty(E;L(H)); \Abs{M}_\infty \leqq 1\} \subseteqq \s H^\infty_{\textnormal{w}}
\]
is compact. If, in addition, $H$ is separable then $B_{\s H^\infty}$ is metrizable and hence sequentially compact. 
\end{thm}
\begin{proof}
 For $\eps\geqq0$ define $B_{\s H(E)}(\eps)\coloneqq \{ f \in \s H(E);\forall z\in E:\abs{f(z)}\leqq \eps\}$. By Montel's theorem $B_{\s H(E)}(\eps)\subseteqq \s H(E)$ is compact $(\eps\geqq 0)$. Observe the following equation to hold
\begin{equation}\label{eq:comp}
   B_{\s H^\infty} = \left( \prod_{\phi,\psi\in H} B_{\s H(E)}\left(\abs{\phi}\abs{\psi}\right)\right)\cap \left\{ M \colon E \to \C^{H\times H}; \forall z\in E: M(z) \text{ sesquilinear}\right\}.
\end{equation}
Indeed, ``$\subseteqq$'' follows from the fact that if $M\in B_{\s H^\infty}$ then for $\phi,\psi\in H$ the map $\langle \phi,M(\cdot)\psi\rangle$ lies in $\s H(E)$ and satisfies the estimate $ \abs{\langle \phi,M(z)\psi\rangle}\leqq \abs{\phi}\abs{\psi}$ for all $z\in E$. Moreover, $M(z)\in L(H)$ and hence $(\phi,\psi) \mapsto \langle\phi,M(z)\psi\rangle$ is sesquilinear for all $z\in E$.
The relation ``$\supseteqq$'' follows from Riesz-Frechet (any bounded sesquilinear mapping $M(z)$ on $H\times H$ is a bounded linear operator on $H$) and a Dunford-type theorem that ensures that local boundedness and weak analyticity for a norming subset of $L(H)'$ is sufficient for analyticity with values in $L(H)$, cf. \cite[p. 139]{Kato1980}.

Now, it is easy to see that \[\left\{ M \colon E \to \C^{H\times H}; \forall z\in E: M(z) \text{ sesquilinear}\right\}\cap \prod_{\phi,\psi\in H} B_{\s H(E)}\left(\abs{\phi}\abs{\psi}\right) \subseteqq \prod_{\phi,\psi\in H} B_{\s H(E)}\left(\abs{\phi}\abs{\psi}\right)\] is closed, if the product is endowed with the product topology. Hence, using the compactness of $B_{\s H(E)}$ and invoking Tikhonov's theorem, we deduce with the help of equation \eqref{eq:comp} that $B_{\s H^\infty}$ is compact as well. 

If we assume that $H$ is separable, the metrizability follows by a standard argument.
\end{proof}

\begin{proof}[Proof of Corollary \ref{thm:acr}] 
The compactness of the embedding $H_1(A)\hookrightarrow H_0(A)$ together with the fact that $D(A)$ is dense in $H_0(A)$ implies the separability of $H_0(A)$. Thus, the sequence $(M_n)_n$ in Corollary \ref{thm:acr} of material laws has a convergent subsequence in $\s H_\textnormal{w}^\infty$, by Theorem \ref{thm:compactnessofmatlaw}. For that subsequence Theorem \ref{thm:ahr} applies. 
\end{proof}

{ We come to the proof of the second part of Theorem \ref{thm:hom-1-dimensional}.
\begin{proof}[Proof of Theorem \ref{thm:hom-1-dimensional} -- part 2: The computation of the limit equation]
The Arzela-Ascoli theorem tells us that $\begin{pmatrix} 0 & \partial_1 \\ \partial_{1,0} & 0\end{pmatrix}$ has compact resolvent in $L_2(0,1)^2$. Now, consider for $k\in \N$ the operator (without loss of generality we assume that $\nu_0>0$ has been chosen so large that the Neumann series expansion is permitted)
\begin{align*}
   M_k\left(\partial_0^{-1}\right) &\coloneqq \begin{pmatrix} \mu(km_1) & 0  \\ 0 & (c(km_1)+D(km_1)\partial_0^\alpha)^{-1} 
                     \end{pmatrix} \\
          & = \begin{pmatrix} \mu(km_1) & 0  \\ 0 & D(km_1)^{-1}\partial_0^{-\alpha}\sum_{\ell=0}^\infty\left(-c(km_1)\partial_0^{-\alpha}D(km_1)^{-1}\right)^\ell  
                     \end{pmatrix}. 
\end{align*}
In particular, for $z\in B(1/(2\nu_0),1/(2\nu_0))$ we get for $k\in\N$
\begin{equation}\label{eq:matpt}
   M_k(z)= \begin{pmatrix} \mu(km_1) & 0  \\ 0 & D(km_1)^{-1}z^\alpha\sum_{\ell=0}^\infty\left(-c(km_1)z^\alpha D(km_1)^{-1}\right)^\ell  
                     \end{pmatrix}.
\end{equation}
Recall from \cite[Theorem 2.6]{CioDon} that for periodic mappings $p\in L_\infty(\R)$ the sequence $(p(n\cdot))_n$ converges to the integral mean over the period in the $\sigma(L_\infty,L_1)$-topology. Thus, the corresponding multiplication operators in $L_2$ converge in the weak operator topology. Hence, for any $z\in B(1/(2\nu_0),1/(2\nu_0))$ we deduce that $(M_k(z))_k$ converges in the weak operator topology in $L(L_2(0,1)^2)$ to the limit $M(z)$ given by
\[
 M(z)= \begin{pmatrix} \int_0^1\mu(x)\dd x & 0  \\ 0 & \int_0^1D(x)^{-1}\dd x z^\alpha+z^\alpha\sum_{\ell=1}^\infty\left(-z^\alpha\right)^{\ell}\int_0^1 \frac{c(x)^\ell}{D(x)^{\ell+1}}\dd x  
                     \end{pmatrix}.
\]
 Using that $L_2(0,1)^2$ is separable and Theorem \ref{thm:compactnessofmatlaw}, we deduce with a standard subsequence argument that $(M_k)_k$ converges to $M$ in $\s H^\infty_{\textnormal{w}}$. Note that the latter particularly implies that $M\in \s H^\infty$. Thus, by Theorem \ref{thm:acr} we get that $(v_k,T_k) \in  H_{\nu,0}(\R;L_2(0,1)^2)$ weakly converges to the solution $(v,T)\in H_{\nu,0}(\R;L_2(0,1)^2)$ of the equation
\[
   \left(\partial_0 \begin{pmatrix} \int_0^1\mu(x)\dd x & 0  \\ 0 & \int_0^1D(x)^{-1}\dd x \partial_0^{-\alpha}+\partial_0^{-\alpha}\sum_{\ell=1}^\infty\left(-\partial_0^{-\alpha}\right)^\ell\int_0^1 \frac{c(x)^\ell}{D(x)^{\ell+1}}\dd x  
                     \end{pmatrix}- \begin{pmatrix} 0 & \partial_1 \\ \partial_{1,0} & 0 
                     \end{pmatrix}\right)\begin{pmatrix} v \\ T 
\end{pmatrix} = \begin{pmatrix} f \\ 0
\end{pmatrix}.
\]
  Reading off the system line-by-line, we get that
\[
   \partial_0 \int_0^1\mu(x)\dd x v -\partial_1 T = f
\]
and
\[
 \partial_0\left(\int_0^1D(x)^{-1}\dd x \partial_0^{-\alpha}+\partial_0^{-\alpha}\sum_{\ell=1}\left(-\partial_0^{-\alpha}\right)^\ell\int_0^1 \frac{c(x)^\ell}{D(x)^{\ell+1}}\dd x  \right)T=\partial_{1,0}v.
\]
Thus, 
\[
 \left(\int_0^1D(x)^{-1}\dd x +\sum_{\ell=1}\left(-\partial_0^{-\alpha}\right)^\ell\int_0^1 \frac{c(x)^\ell}{D(x)^{\ell+1}}\dd x  \right)T=\partial_0^{\alpha}\partial_{1,0}\partial_0^{-1}v.
\]
Now, without restriction we assume that $\nu$ has been chosen large enough such that
\begin{align*}
 &\left(\int_0^1D(x)^{-1}\dd x +\sum_{\ell=1}\left(-\partial_0^{-\alpha}\right)^\ell\int_0^1 \frac{c(x)^\ell}{D(x)^{\ell+1}}\dd x  \right)^{-1} \\
 & =\left(\int_0^1D(x)^{-1}\dd x\right)^{-1}\left(1+ \left(\int_0^1D(x)^{-1}\dd x\right)^{-1}\sum_{\ell=1}\left(-\partial_0^{-\alpha}\right)^\ell\int_0^1 \frac{c(x)^\ell}{D(x)^{\ell+1}}\dd x  \right)^{-1} \\
 & =\left(\int_0^1D(x)^{-1}\dd x\right)^{-1}\sum_{k=0}^\infty \left(-\left(\int_0^1D(x)^{-1}\dd x\right)^{-1}\sum_{\ell=1}\left(-\partial_0^{-\alpha}\right)^\ell\int_0^1 \frac{c(x)^\ell}{D(x)^{\ell+1}}\dd x\right)^k
\end{align*}
exists as a Neumann series expression. Substituting $u\coloneqq\partial_0^{-1}v$ we arrive at the asserted limit expression. The proof is finished. 
\end{proof}
}

{ Before ending this section with a proof of Theorem \ref{thm:homparticular}, we recall \cite[Lemma 4.1]{Waurick2012Asy} (also see \cite[Lemma 3.2.2]{Waurick2011Diss}) needed in the proof.

\begin{lem}\label{Le: Nullspace_away} Let $H_1,H_2$ be two Hilbert spaces, $C:D(C)\subseteqq H_1\to H_2$ be a densely defined, closed linear operator. Assume that $(D(C),\abs{\cdot}_C)\hookrightarrow (H_1,\abs{\cdot}_{H_1})$ is compact. Then $(D(C^*)\cap N(C^*)^{\bot},\abs{\cdot}_{C^*}) \hookrightarrow (H_2,\abs{\cdot}_{H_2})$ is compact.
\end{lem}
\begin{proof}
 The proof rests on the theorem of the polar decomposition \cite[p.334]{Kato1980} for densely defined, closed operators.
\end{proof} }
%

\begin{proof}[Proof of Theorem \ref{thm:homparticular}] 
  Since the domain of the operator $\Grad$ ($\Grad_0$, resp.) endowed with the graph norm is assumed to be compactly embedded into $L_2(\Omega)^n$, we deduce with the help of { Lemma \ref{Le: Nullspace_away}} the compactness of the embedding  $H_1(A_\textnormal{red})\hookrightarrow H_0(A_\textnormal{red})$. The boundedness of 
\[
   \left(M_n(\partial_0^{-1})\right)_n=\left( \begin{pmatrix} \mu_{n} & 0  \\ 0 & (\Pi_A(C_{n}+D_{n}\partial_0^\alpha)\Pi_A^*)^{-1} 
                     \end{pmatrix}\right)_n
\]
in $\s H^\infty_\textnormal{w}$ ensures the applicability of Corollary \ref{thm:acr}. The limit expression follows by a Neumann series expansion { similar to the one in the proof of Theorem \ref{thm:hom-1-dimensional} part 2.}
\end{proof}

\section{The proof of Theorem \ref{thm:ahr} and the weak-strong principle}\label{sec:wsp}

The method to prove Theorem \ref{thm:ahr} is similar to the one used in the proof of \cite[Theorem 3.5]{Waurick2012Asy}. In order to follow this strategy here, we have to derive a stronger version of \cite[Theorem 2.4]{Waurick2012Asy}, the so-called 'weak-strong-principle'. In \cite{Waurick2012Asy} the property of the material laws being analytic at $0$ was used. This assumption has to be weakened due to the fact that fractional \mbox{(time-)}derivatives do not correspond to a material law being analytic at $0$. The strengthened weak-strong principle result Theorem \ref{Th: weal-strong_1} does not need this assumption anymore. Consequently, the proof of Theorem \ref{Th: weal-strong_1} is completely different to the one in \cite{Waurick2012Asy}. In order to motivate Theorem \ref{Th: weal-strong_1}, we recall a fact relating convergent material laws and the convergence of the associated operators arising in the aforementioned functional calculus for $\partial_{0}^{-1}$. 

\begin{lem}[{\cite[Lemma 3.5]{WaurickMMA}}] Let $H$ be a Hilbert space, $\nu>0$, $r>\frac{1}{2\nu}$. Let $(M_n)_n$ be a bounded sequence in $\s H^\infty_{\textnormal{w}}(B(r,r);L(H))$ converging to $M\in \s H^\infty_{\textnormal{w}}$. Then $(M_n(\partial_{0}^{-1}))_n$ converges in the weak operator topology of $L(H_{\nu,0}(\R;H))$ to $M(\partial_{0}^{-1})$. 
\end{lem}

In the light of the previous result one might wonder if $(v_n)_n$ converges weakly in $H_{\nu,0}(\R;H)$ and $(M_n)_n$ converges in $\s H^\infty_{\textnormal{w}}$ that then also 
\[
 \lim_{n\to\infty} M_n(\partial_{0}^{-1})v_n = \lim_{n\to\infty}M_n(\partial_{0}^{-1})\textnormal{w-}\lim_{n\to\infty} v_n.
\]
  Easy examples show that this equality does not hold in general. Hence, suitable assumptions are needed to deduce such an equality. Our version of the weak-strong principle reflects this fact.

\begin{thm}[weak-strong principle]\label{Th: weal-strong_1} Let $H,H_1$ be Hilbert spaces, and such that $H_1\hookrightarrow\hookrightarrow H$. Let $\nu_0>0, r>1/(2\nu_0)$ and $(v_n)_n$ be a weakly convergent sequence in $H_{\nu_0,1}(\R;H_1)$ with limit $v\in H_{\nu_0,1}(\R;H_1)$ and assume that $\inf_{n\in\N} \inf \spt v_n >-\infty$. Let $(M_n)_n$ be a bounded and convergent sequence in $\s H^\infty_\textnormal{w} (B(r,r);L(H))$ with limit $M$. Then for all $\nu>\nu_0$ and $f\in H_{\nu,-1}(\R)\cong H_{\nu,1}(\R)^*$ we have
\[
    f(M_n(\partial_{0}^{-1})v_n)\stackrel{n\to\infty}{\rightharpoonup} f(M(\partial_{0}^{-1})v)\in H.
\]
Moreover, $M_n(\partial_{0}^{-1})v_n \stackrel{n\to\infty}{\rightharpoonup} M(\partial_{0}^{-1})v \in H_{\nu,1}(\R;H)$.
\end{thm}
{
\begin{rem}\label{rem:seq_conv_itself}
  Note that from the assertion of Theorem \ref{Th: weal-strong_1} it also follows that $M_n(\partial_0^{-1})v_n\stackrel{n\to\infty}{\rightharpoonup} M(\partial_{0}^{-1})v \in H_{\nu_0,1}(\R;H)$. Indeed, the assumptions in Theorem \ref{Th: weal-strong_1} guarantee that $\left(M_n(\partial_0^{-1})v_n\right)_n$ is a bounded sequence in $H_{\nu_0,1}(\R;H)$. Thus, we get weak convergence with the same argument as in Remark \ref{rem:implicitly}(c). 
\end{rem}
}

The proof of Theorem \ref{Th: weal-strong_1} needs some preliminaries. We need the concept of tensor products of Hilbert spaces. As a general reference we refer to the monograph \cite{Tensor}, a short overview of the concepts needed is also given in \cite[Section 1.2.3]{Picard} or \cite[Appendix B.1 and B.2]{Waurick2011Diss}. We recall some notation. Let $H,H_1,H_2$ be Hilbert spaces. We denote the Hilbert space tensor product of $H_1$ and $H_2$ by $H_1\otimes H_2$. For any densely defined closed linear operator $A\colon D(A)\subseteqq H_1 \to H_2$ we denote the canonical extension of $A$ as a mapping from the product space $H\otimes H_1$ to $H\otimes H_2$ by $1_H\otimes A $. Similarly, write $A\otimes 1_H$ for the respective extension from $H_1\otimes H$ to $H_2\otimes H$. Note that if $H=L_2(\mu)$ for some $\sigma$-finite measure space $(\Omega,\mu)$ then $H\otimes H_1\cong L_2(\mu;H_1)$. We use (and have used) this identification throughout without further notice. A first elementary observation is the following.

\begin{lem}\label{Le: weal-strong_0} Let $H,H_1,H_2$ be Hilbert spaces and assume that $U:H_1\to H_2$ is unitary. Let $g\in H_2^*$ and denote the dual operator of $U$ by $U'\colon H_2^*\to H_1^*$. Then the following holds:
 If $(v_n)_n$ is a sequence in $H_1\otimes H$, $v\in H_1\otimes H$ and such that for all $f\in H_1^*$ we have $(f\otimes 1_H)(v_n)\rightharpoonup (f\otimes 1_H) (v) \in H$ as $n\to \infty$, then $(g\otimes 1_H)((U\otimes 1_H)(v_n)) \rightharpoonup (g\otimes 1_H) ((U\otimes 1_H)(v))\in H$ as $n\to\infty$.
\end{lem}
\begin{proof}
We have $(g\otimes 1_H)((U\otimes 1_H)(v_n))=(U'g\otimes 1_H)(v_n) \rightharpoonup (U'g\otimes 1_H)(v)= (g\otimes 1_H) ((U\otimes 1_H)(v))\in H$ as $n\to\infty$.
\end{proof}

The second lemma relates weak convergence in $H_{\nu,0}(\R;H)$ to pointwise (weak) convergence of the Fourier-Laplace transforms. 

\begin{lem}\label{Le: Laplace_ptw} Let $H$ be a Hilbert space, $\nu_0>0$ and let $(v_n)_n$ be a weakly convergent sequence in $H_{\nu_0,0}(\R;H)$ with limit $v$. Assume, in addition, that $s\coloneqq \inf_{n\in\N}\inf \spt v_n >-\infty$. Then for $t\in \R$ and $\nu>\nu_0$
\[
    \s L_{\nu} v_n(t) \rightharpoonup \s L_{\nu} v(t) \in H\quad (n\to\infty).
\]
Moreover, $(t\mapsto \s L_{\nu} v_n(t))_n$ is bounded in $C_b(\R;H)$, the space of bounded continuous functions from $\R$ to $H$.
\end{lem}
\begin{proof}
   Note that due to the support constraint on the sequence $(v_n)_n$ the sequence $(x\mapsto e^{-\nu x}v_n(x))_n=(x\mapsto e^{-(\nu-\nu_0)x}e^{-\nu_0 x}v_n(x))_n$ is bounded in $L_1(\R;H)$. This establishes the second assertion. Furthermore, for $\phi\in H$, $n\in \N$, $t\in \R$,
\begin{align*}
   \langle \s L_{\nu} v_n(t),\phi \rangle & = \left\langle \frac{1}{\sqrt{2\pi}} \int_{\R} e^{-ixt}e^{-\nu x}v_n(x)\dd x, \phi \right\rangle \\
                                          & = \frac{1}{\sqrt{2\pi}} \int_{\R} \left\langle e^{-ixt}e^{-\nu x}v_n(x) , \phi \right\rangle \dd x \\
                                          & = \frac{1}{\sqrt{2\pi}} \int_{\R} \left\langle v_n(x) ,e^{ixt}e^{-(\nu-2\nu_0) x}\chi_{\R_{>s}}(x) \phi \right\rangle e^{-2\nu_0x} \dd x\\
                                          & = \langle v_n , \frac{1}{\sqrt{2\pi}}e^{i(\cdot)t}e^{-(\nu-2\nu_0) (\cdot)}\chi_{\R_{>s}}(\cdot) \phi\rangle_{H_{\nu_0,0}(\R;H)}\\
                                        & \tor{n}{\infty} \langle v , \frac{1}{\sqrt{2\pi}}e^{i(\cdot)t}e^{-(\nu-2\nu_0) (\cdot)}\chi_{\R_{>s}}(\cdot) \phi\rangle_{H_{\nu_0,0}(\R;H)}                                         =\langle \s L_{\nu} v(t),\phi \rangle.\qedhere
\end{align*}
\end{proof}


\begin{proof}[Proof of Theorem \ref{Th: weal-strong_1}] For the first assertion, by Lemma \ref{Le: weal-strong_0}, it suffices to prove that 
\[
    \langle g, M_n\left(\frac{1}{im+\nu}\right) \s L_\nu v_n \rangle \rightharpoonup \langle g, M\left(\frac{1}{im+\nu}\right)\s L_\nu v\rangle \in H \quad (n\to\infty) 
\]
for all $g\in H_{-1}(im +\nu)=H_{1}(im +\nu)^*$. For this note that $(\s L_{\nu} v_n)_n$ weakly converges in $H_{1}(im+\nu)\otimes H_1$ by the unitarity of $\s L_\nu$, particularly, it is a bounded sequence in $H_{1}(im+\nu)\otimes H$. The boundedness of $\left(M_n\left(\frac{1}{im+\nu}\right) \right)_n$ ensures the boundedness of $\left(M_n\left(\frac{1}{im+\nu}\right)\s L_{\nu} v_n\right)_n$ in $H_{1}(im+\nu)\otimes H$. Hence, we find a weakly convergent subsequence for which we re-use the indices $n$. It remains to identify the limit. For this, we deduce by $H_{\nu,1}(\R;H_1)\hookrightarrow H_{\nu,0}(\R;H_1)$ and Lemma \ref{Le: Laplace_ptw} that for all $t\in\R$ the sequence $(\s L_{\nu} v_n(t))_n$ converges weakly to $\s L_{\nu}v(t)$ in $H_1$. The compact embedding $H_1\hookrightarrow \hookrightarrow H$ gives strong convergence of $(\s L_{\nu} v_n(t))_n$ to $\s L_{\nu} v(t)$ in $H$. Thus, for all $t\in \R$
\[
   M_n\left(\frac{1}{it+\nu}\right) \s L_{\nu} v_n (t) \rightharpoonup M\left(\frac{1}{it+\nu}\right)\s L_{\nu} v(t) \in H.
\]
 Let $g\in H_{-1}(im+\nu)(\subseteqq L_{1,\textnormal{loc}}(\R))$ be bounded and with bounded support. Then by Lebesgue's dominated convergence theorem  
\begin{align*}
   \left\langle g,M_n\left(\frac{1}{im+\nu}\right) \s L_\nu v_n \right\rangle & = \int_\R g(t)^* \cdot M_n\left(\frac{1}{it+\nu}\right) \s L_{\nu} v_n (t)\dd t \\
     & = \int_\R \left(\frac{1}{-it+\nu}g(t)\right)^* M_n\left(\frac{1}{it+\nu}\right) (it+\nu)\s L_{\nu} v_n (t)\dd t \\
    &\rightharpoonup \int_\R \left(\frac{1}{-it+\nu}g(t)\right)^* M\left(\frac{1}{it+\nu}\right) (it+\nu)\s L_{\nu} v (t)\dd t \in H, \end{align*}
as $n\to\infty$ (Note that $\left(t\mapsto (it+\nu)\s L_{\nu} v_n (t)\right)_n$ is locally bounded by Lemma \ref{Le: Laplace_ptw}.) The density of bounded functions with bounded support in $H_{-1}(im+\nu)$ identifies the limit of $\left(M_n\left(\frac{1}{im+\nu}\right) \s L_{\nu} v_n\right)_n$ being equal to $M\left(\frac{1}{im+\nu}\right) \s L_{\nu} v$. 

 The last assertion of Theorem \ref{Th: weal-strong_1} follows again with a subsequence argument: Let $(n_k)_k$ be a subsequence of $(n)_n$. Then, since $\left(M_{n_k}(\partial_{0}^{-1})v_{n_k}\right)_k$ is bounded in $H_{\nu,1}(\R;H)$, there exists a weakly convergent subsequence, again denoted with the index sequence $(n_k)_k$. Denote its limit by $w\in H_{\nu,1}(\R;H)$. For $t\in \R$ let $\delta_t \in H_{\nu,-1}(\R)$ be the evaluation functional at $t\in\R$ ($\delta_t$ is indeed an element of $H_{\nu,-1}(\R)$ due to a Sobolev embedding result, see e.g.\ \cite[Lemma 3.1.59]{Picard} or \cite[Remark 2.1(ii)]{Waurick2012Asy}). Observing that $\delta_t$ maps weakly convergent sequences in $H_{\nu,1}(\R;H)$ to weakly convergent sequences in $H$, we arrive at
\[
  \delta_t\left(M_{n_k}(\partial_{0}^{-1})v_{n_k}\right) \rightharpoonup w(t) \in H
\]
and, by our results above, 
\[
   \delta_t\left(M_{n_k}(\partial_{0}^{-1})v_{n_k}\right) \rightharpoonup \left(M(\partial_{0}^{-1})v\right)(t) \in H.
\]
Hence, $w(t)= \left(M(\partial_{0}^{-1})v\right)(t)$ for all $t\in\R$. This shows the assertion.
\end{proof}


There is yet another useful lemma helping to conclude the proof of Theorem \ref{thm:ahr}:

\begin{lem}\label{le:fatou} Let $H$ be a Hilbert space, $(f_n)_n$ a weakly convergent sequence in $L_2(\R;H)$. Assume, in addition, that $f_n(t)\rightharpoonup f(t)$ in $H$ for a.e. $t\in\R$ and $n\to\infty$. Then for all $a\in\R$ and $n\in\N$, we have that
\[
    \int_{-\infty}^a \abs{f_n(t)-f(t)}^2 \dd t \leqq 4\sup_{n\in\N} \int_{-\infty}^a \abs{f_n(t)}^2\dd t. 
\] 
\end{lem}
\begin{proof}
 Let $a\in\R$ and $n\in\N$. For almost every $t\in\R$ we have the (standard) estimate $\abs{f(t)}\leqq \liminf_{n\to\infty}\abs{f_n(t)}$. Hence, we get, employing Fatou's lemma, that
\begin{align*}
   \int_{-\infty}^a \abs{f_n(t)-f(t)}^2\dd t &\leqq 2\int_{-\infty}^a \abs{f_n(t)}^2+\abs{f(t)}^2\dd t \\
                                             &\leqq 2\int_{-\infty}^a \abs{f_n(t)}^2+\liminf_{n\to\infty}\abs{f_n(t)}^2\dd t\\
                                             &\leqq 2\left(\sup_{n\in\N}\int_{-\infty}^a \abs{f_n(t)}^2\dd t +\liminf_{n\to\infty}\int_{-\infty}^a \abs{f_n(t)}^2\dd t\right)\\
                                             &\leqq 4\sup_{n\in\N} \int_{-\infty}^a \abs{f_n(t)}^2\dd t. \qedhere
\end{align*}
\end{proof}

Now, we have all the tools at hand to prove Theorem \ref{thm:ahr}. The principle idea of the proof is the same as in \cite[Theorem 3.5]{Waurick2012Asy}. Therefore, we will not go too much into the details for the arguments already used in \cite{Waurick2012Asy}. However, there are several adjustments needed to conclude the proof of Theorem \ref{thm:ahr}. We give the main ideas as follows.

\begin{proof}[Proof of Theorem \ref{thm:ahr}] 
 To begin with, let $f\in H_{\nu,0}(\R;H)$ and such that $\inf \spt f >-\infty$. Now, let $u_n\in H_{\nu,0}(\R;H)$ solve
{ \begin{equation}\label{eq:proof_ahr}
    \left(\partial_0 M_n\left(\partial_{0}^{-1}\right)+A\right)u_n = f.
\end{equation}}
By causality, it follows that $\inf_{n\in\N}\inf \spt u_n > -\infty$. Moreover, { by Theorem \ref{thm:solth} it follows that $(u_n)_n$ is a bounded sequence in $H_{\nu,0}(\R;H)$. Furthermore, from \eqref{eq:proof_ahr} we read off
\[
  Au_n = \partial_0 M_n\left(\partial_0^{-1}\right) u_n + f \in H_{\nu,-1}(\R;H).
\]
Thus, estimating the right-hand side and using the boundedness of $(u_n)_n$ in $H_{\nu,0}(\R;H)$, we deduce that 
\begin{align*}
 \abs{Au_n}_{H_{\nu,-1}(\R;H)} & =\abs{\partial_0 M_n\left(\partial_0^{-1}\right) u_n + f}_{H_{\nu,-1}(\R;H)} \\
                               & \leqq \abs{M_n\left(\partial_0^{-1}\right) u_n }_{H_{\nu,0}(\R;H)}+\abs{f}_{H_{\nu,-1}(\R;H)} \\
                               & \leqq \sup_{n\in\N}\Abs{M_n(\partial_0^{-1})}\sup_{n\in\N} \abs{u_n}_{H_{\nu,0}(\R;H)}+\abs{f}_{H_{\nu,-1}(\R;H)}\\
                               &<\infty.
\end{align*}
Thus, $(u_n)_n$ is a bounded sequence in $H_{\nu,-1}(\R;H_1(A))$.
} There is a weakly convergent subsequence $(u_{n_k})_k$ { in $H_{\nu,0}(\R;H)\cap H_{\nu,-1}(\R;H_1(A))$} with limit $u$. { In consequence, $(\partial_0^{-2}u_{n_k})_k$ is a weakly convergent sequence in $H_{\nu,1}(\R;H_1(A))$ as $\partial_0^{-2}$ is a unitary transformation from $H_{\nu,-1}(\R;H_1(A))$ onto $H_{\nu,1}(\R;H_1(A))$. Now, $\inf_n \inf\spt \partial_0^{-2}u_n>-\infty$, by causality of $\partial_0^{-2}$. Thus, by Theorem \ref{Th: weal-strong_1} and Remark \ref{rem:seq_conv_itself}, we get that $M_{n_k}(\partial_0^{-1})\partial_0^{-2}u_{n_k} \stackrel{k\to\infty}{\rightharpoonup}M(\partial_0^{-1})\partial_0^{-2}u$ in $H_{\nu,1}(\R;H)$. Hence, using that $\partial_0^3$ commutes with functions of $\partial_0^{-1}$ and that $\partial_0^{3}$ can be considered as a unitary mapping from $H_{\nu,1}(\R;H)$ to $H_{\nu,-2}(\R;H)$, we get that $\text{w-}\lim_{k\to\infty}\partial_0 M_{n_k}\left(\partial_{0}^{-1}\right)u_{n_k}=\partial_0 M\left(\partial_0^{-1}\right)u$ in $H_{\nu,-2}(\R;H)$.}  Realizing that $(Au_{n_k})_k$ weakly converges to $Au$, we arrive at
\[
   \left(\partial_0 M\left(\partial_{0}^{-1}\right)+A\right)u = f.
\]
The well-posedness of the latter equation yields uniqueness of the limit of $(u_{n_k})_k$. Thus, $(u_n)_n$ weakly converges in $H_{\nu,0}(\R;H)$.
A density argument shows that $\left(\partial_0 M_n\left(\partial_{0}^{-1}\right)+A\right)^{-1}$ converges in the weak operator topology of $L(H_{\nu,0}(\R;H))$ to $\left(\partial_0 M\left(\partial_{0}^{-1}\right)+A\right)^{-1}$.

 Let $(f_n)_n$ be a weakly convergent sequence in $H_{\nu,0}\cap H_{\nu_1,0}(\R;H)$ for some $\nu>\nu_1>1/(2r)$ satisfying the integrability condition given in the theorem. Denote its limit by $f$. Now, let $u_n\in H_{\nu,0}(\R;H)$ solve
\[
   \left(\partial_0 M_n\left(\partial_{0}^{-1}\right)+A\right)u_n = f_n. 
\]
Following an idea given in \cite[proof of Lemma 3.2]{Schweizer}, we decompose $u_n = v_n + w_n$ in the way that $v_n$ and $w_n$ are the solutions of the respective equations
\[
   \left(\partial_0 M_n\left(\partial_{0}^{-1}\right)+A\right)v_n = f,
\]
and
\[
   \left(\partial_0 M_n\left(\partial_{0}^{-1}\right)+A\right)w_n = f_n-f,
\]
for all $n\in\N$.

We already know that $(v_n)_n$ weakly converges to the solution $u$ of the following problem
\[
   \left(\partial_0 M\left(\partial_{0}^{-1}\right)+A\right)u = f.
\]
We are left with showing that $(w_n)_n$ (weakly) converges to $0$. For this, observe that we have that $(\partial_0^{-2}w_n)_n$ is bounded in $H_{\nu,1}(\R;H_1(A))$. { Thus, there exists a weakly convergent subsequence of $(\partial_0^{-2}w_n)_n$ in $H_{\nu,1}(\R;H_1(A))$. We re-use the index sequence $(n)_n$ for this subsequence. Recall that, by the Sobolev embedding theorem (see e.g.~\cite[Lemma 3.1.59]{Picard} or \cite[Lemma 5.2]{Kalauch}), the operator $\delta_t$ of point evaluation at time $t\in \R$ is a continuous mapping from $H_{\nu,1}(\R;H_1(A))$ to $H_1(A)$. Hence, $\delta_t$ maps weakly convergent sequences to weakly convergent sequences. In particular, for every $t\in \R$ the sequence $(\delta_t\left(\partial_0^{-2}w_n\right))_n=(\partial_0^{-2}w_n(t))_n$ weakly converges in $H_1(A)$. By compactness of the embedding $H_1(A)\hookrightarrow H_0(A)=H$, weakly convergent sequences in $H_1(A)$ are mapped to strongly convergent sequences in $H_0(A)$. Hence, $(\partial_0^{-2}w_n(t))_n$ strongly converges in $H_0(A)$.} Moreover, let $C>0$ be an upper bound for $(\abs{\partial_0^{-2}w_n}_{\nu,0})_n$.  

{ For $n\in\N$, we compute with the help of the strict monotonicity of $ \left(\partial_0 M_n\left(\partial_{0}^{-1}\right)+A\right)$ (recall Remark \ref{rem:solth}(b)) for $a\in\R$:
\begin{align*}
   c\abs{\partial_0^{-2}w_n}& \leqq \abs{\langle (\partial_0 M_n(\partial_0^{-1})+A)\partial_0^{-2}w_n,\partial_0^{-2}w_n\rangle}\\
                             & = \abs{ \langle \partial_0^{-2}(f_n - f),\partial_0^{-2} w_n\rangle }\\
                             & \leqq  \int_{-\infty}^a \abs{\partial_0^{-2}(f_n-f)(t)}\abs{\partial_0^{-2}w_n(t)}e^{-2\nu t}\dd t \\ &\quad+ 
                                         \int_{a}^\infty \abs{\langle\partial_0^{-2}(f_n-f)(t),\partial_0^{-2}w_n(t)\rangle}e^{-2\nu t}\dd t.
\end{align*}
Further, we get with Lemma \ref{le:fatou} applied to the sequence $\left(e^{-\nu \cdot}\partial_0^{-2}f_n(\cdot)\right)_n$ that 
\begin{align*}
   & \int_{-\infty}^a \abs{\partial_0^{-2}(f_n-f)(t)}\abs{\partial_0^{-2}w_n(t)}e^{-2\nu t}\dd t \\ 
                             & \leqq \abs{\chi_{(-\infty,a]}(m_0)\partial_0^{-2}(f_n-f)}_{\nu,0}\abs{\chi_{(-\infty,a]}(m_0)\partial_0^{-2}w_n}_{\nu,0} \\
&\leqq C\abs{\chi_{(-\infty,a]}(m_0)\partial_0^{-2}(f_n-f)}_{\nu,0} \\
&\leqq 2C\sup_{n\in\N}\abs{\chi_{(-\infty,a]}(m_0)\partial_0^{-2}\chi_{(-\infty,a]}(m_0)f_n}_{\nu,0} \\
&\leqq \frac{2C}{\nu^2}\sup_{n\in\N}\abs{\chi_{(-\infty,a]}(m_0)f_n}_{\nu,0}
\end{align*}
where we also used causality of $\partial_0^{-2}$. The latter computations together with
\begin{multline*}\int_{a}^\infty \abs{\langle \partial_0^{-2}(f_n-f)(t),\partial_0^{-2}w_n(t)\rangle}e^{-2\nu t}\dd t\\ =\int_{a}^\infty \abs{\langle\partial_0^{-2}(f_n-f)(t)e^{-\nu_1 t},\partial_0^{-2}w_n(t)e^{-\nu_1 t}\rangle}e^{-2(\nu-\nu_1) t}\dd t.\end{multline*}
yield
}\begin{multline} \label{est:last}
   c\abs{\partial_0^{-2}w_n}\leqq \frac{2C}{\nu^2}\sup_{n\in\N}\abs{\chi_{(-\infty,a]}(m_0)f_n}_{\nu,0} \\
 + \int_{a}^\infty \abs{\langle\partial_0^{-2}(f_n-f)(t)e^{-\nu_1 t},\partial_0^{-2}w_n(t)e^{-\nu_1 t}\rangle}e^{-2(\nu-\nu_1) t}\dd t.
\end{multline}
By \cite[Lemma 3.1.59]{Picard} or \cite[Lemma 5.2]{Kalauch}, i.e., the continuity of the embedding \[
H_{\nu_1,1}(\R;H)\hookrightarrow C_{\nu_1}(\R;H)\coloneqq \{f\in C(\R;H); \sup_{t\in\R}\abs{f(t)e^{-\nu_1 t}}_H<\infty\},
                                          \]
the expression $(t\mapsto \abs{\partial_0^{-2}(f_n-f)(t)}e^{-\nu_1 t}\abs{\partial_0^{-2}w_n(t)e^{-\nu_1 t}})_n$ is uniformly bounded in both $n$ and $t$. Hence, by Lebesgue's dominated convergence theorem, the last term on the right-hand side in \eqref{est:last} tends to $0$ for every given $a\in\R$. Thus, $(w_n)_n$ converges to $0$ since the first term on the right-hand side in \eqref{est:last} tends to zero if $a\to-\infty$. 
\end{proof}

\section*{Acknowledgements}
The author wants to thank Sascha Trostorff and J\"urgen Voigt for useful discussions particularly concerning Theorem \ref{Th: weal-strong_1}. Moreover, the author thanks Daniel Karrasch for drawing the author's attention to reference \cite{Schweizer}.

\end{document}